\documentclass[10pt,leqno,a4paper]{article}
\usepackage{a4wide}
\usepackage{color}
\usepackage{amsfonts}
\usepackage{amsmath, amsfonts, amsthm, amssymb, amscd}
\usepackage[mathcal]{euscript}
\usepackage{mathrsfs}
\usepackage{dsfont}
\usepackage{ulem}
\usepackage{bbm}
\usepackage{graphicx}

\usepackage{slashed}

\newtheorem{theorem}{Theorem}[section]
\newtheorem{proposition}[theorem]{Proposition}
\newtheorem{lemma}[theorem]{Lemma}

\newtheorem{remark}[theorem]{Remark}

\numberwithin{equation}{section}

\usepackage{multirow}
\usepackage{tabularx}
\usepackage{lscape}

\newcommand \Kcal {\mathcal K}

\newcommand \Fcal{\mathcal{F}}
\newcommand \Hcal {\mathcal H}

\newcommand \Ecal {\mathcal E}

\newcommand \xb {\bar {x}}
\newcommand \delb {\bar {\del}}

\newcommand \Boxt {\widetilde {\Box}}
\newcommand \del \partial
\newcommand \delu {\uline{\del}}

\newcommand \Tu {\uline{T}}

\newcommand \minu{\uline{m}}
\newcommand \Au{\uline{A}}

\newcommand \Psiu{\uline{\Psi}}
\newcommand \Phiu{\uline{\Phi}}

\newcommand \Ec {E_2}

\newcommand \RR{\mathbb{R}}

\newcommand {\vep}{\varepsilon}

\newcommand {\dels}{\slashed{\del}}

\makeatletter
\def\hlinew#1{%
  \noalign{\ifnum0=`}\fi\hrule \@height #1 \futurelet
   \reserved@a\@xhline}
\makeatother

\title{Global solutions of nonlinear wave-Klein-Gordon system in two spatial dimensions: A prototype of strong coupling case \footnote{The present work belongs to a research project ``Global stability of quasilinear wave-Klein-Gordon system in $2 + 1$ space-time dimension'' (11601414), supported by NSFC.}}

\author{Yue MA \footnote{School of Mathematics and Statistics, Xi'an Jiaotong University, Xi'an, Shaanxi 710049, P.R. China.\ E-mail: yuemath@xjtu.edu.cn}}

\begin{document}

\maketitle

\begin{abstract}
	In this article we will develop some techniques aimed at the strong couplings in two-dimensional wave-Klein-Gordon system. We distinguish the roles of different type of decay factors and develop a method  which permits us to ``exchange'' one type of decay into the other. Then a global existence result of a model problem is established. We also give a sketch of the Klein-Gordon-Zakharov model system and establish the associate global existence result. 
\end{abstract}

\section{Introduction}
\subsection{The Model problems}
This article belongs to a research project in which we attempt to understand the effects of different quadratic terms coupled in diagonalized wave-Klein-Gordon system in $2+1$ dimensional space-time. In the previous works for example \cite{M2} and \cite{Dong-2020}, we mainly concentrate on the so-called weak coupling cases, i.e., in the wave equation there is no pure Klein-Gordon terms. In the present work we start an investigation on the strong coupling case. We will develop several technical tools and establish the global existence result for the following model system:
\begin{equation}\label{eq-main}
\left\{
\aligned
&\Box u = A_1^{\alpha\beta}\del_{\alpha}u\del_{\beta}u + A_3^{\alpha\beta}\del_{\alpha}u\del_{\beta}v + A_4^{\alpha}v\del_{\alpha}u
\\
&\quad\quad\quad + B_2^{\alpha\beta}\del_{\alpha}v\del_{\beta}v + B_3^{\alpha}v\del_{\alpha}v + K_2v^2,
\\
&\Box v + c^2v = A_5^{\alpha\beta}\del_{\alpha}u\del_{\beta}u.
\endaligned
\right.
\end{equation}
In \eqref{eq-main}, we remark the general strong coupling terms $B_2^{\alpha\beta}\del_{\alpha}v\del_{\beta}v + B_3^{\alpha}v\del_{\alpha}v + K_2v^2$. The quadratic form $A_1, A_3$ and $A_5$ are supposed to be null. The rest are constant-coefficient (multi-)linear forms. In fact in the wave equation we have included all possible quadratic semi-linear terms on $(\del u,\del v, v)$. Counterintuitively, quasi-linear terms in wave equation are much easier that semi-linear ones due to the so called Hessian structure which will be analyzed when we regard \eqref{eq1-Zakharov} ,see in detail  below. So we exclude them in \eqref{eq-main}. 

In the mean time, we will show an other application of these techniques which is
the following model system formulated form the Klein-Gordon-Zakharov system in $\RR^{2+1}$ introduced and studied in \cite{Dong-2020-2}:
\begin{equation}\label{eq1-Zakharov}
\left\{
\aligned
&\Box E + E = E\Delta u ,
\\
&\Box u = |E|^2,
\endaligned
\right.
\end{equation} 
where $u: \RR^{2+1}\rightarrow \RR$ a scalar and $E: \RR^{2+1}\rightarrow \RR^2$ a vector. We will give an alternative approach to the global existence of this system.

For clarity the initial data are supposed to be compactly supported on the initial hyperboloid $\Hcal_2$ or equivalently, on $\{t=2\}$. This is not an essential restriction because with the so-called Euclidean-hyperboloidal foliation (see for example \cite{M-2018} for an one-dimensional case), the argument here can be easily generalized to non-compactly-supported initial data with sufficient spatial decreasing rate. 

As explained in many existing works, the problem of global existence of wave-Klein-Gordon system is more delicate in $2+1$ dimension than in $3+1$ dimensional case because of the slow decay rate of both wave and Klein-Gordon equations in lower dimension. We recall \cite{Stingo-2018}, \cite{Stingo-2019} for the methods based on Fourier analysis (which are in non-diagonalized quasi-linear case) and \cite{Dong-2020}, \cite{Dong-2020-2} for the analysis in physical space-time. 
It also worth to mention the following results \cite{Godin-1993}, \cite{A1}, \cite{A2}, \cite{Cai-2018}, \cite{Zha-2019}, \cite{Hou-2020} on wave equations in $\RR^{2+1}$ and \cite{Dfx}, \cite{KS-2011} on Klein-Gordon equations in $\RR^{2+1}$ dimension.

The main challenge of \eqref{eq-main} comes from the insufficiency of the so-called principle decay, which is due to the strong coupling terms and the interaction terms $\del_{\alpha}u\del_{\beta}v, v\del_{\alpha}u$ coupled in wave equation. The objective of the present work is mainly concentrated on this difficulty. We will give a detailed explanation in the coming two subsections.

\subsection{Strong couplings v.s. weak couplings}
The systems with strong coupling arise naturally in many physical or geometrical context. For example the Einstein-Klein-Gordon system
\begin{equation}
\aligned
&\Boxt_g g_{\alpha\beta} = F_{\alpha\beta}(g;\del g,\del g) + T_1(\del\phi,\del\phi) + T_2(\phi,\phi),
\\
&\Boxt_g\phi + c^2\phi = 0
\endaligned
\end{equation} 
where $T_1,T_2$ are quadratic forms. The wave map system formulated in \cite{Ab-2019}
\begin{equation}\label{eq5-08-09-2020}
\aligned
&\Box u = -2\sum_{k=2}^n\phi^k \del_1\phi^k + \text{cubic terms}
\\
&\Box \phi^k + \phi^k = 2\phi^k \del_1u + \text{cubic terms}, \quad k=2,\cdots n.
\endaligned
\end{equation}
as well as the already mentioned Klein-Gordon-Zakharov model system \eqref{eq1-Zakharov}.

We will explain the challenge comes form the strong coupling. For the convenience of discussion we recall some notation. We are working in $\RR^{2+1} = \{(t,x)|t\in\RR, x\in\RR^2\}$ and recall that
$$
\Kcal = \{t>r+1\},\quad s = \sqrt{t^2-r^2} = \sqrt{t^2-|x|^2}, \quad \Hcal_s = \{(t,x)| t = \sqrt{s^2+r^2}\}.
$$
For a wave-Klein-Gordon system with unit propagation speed, if the initial data are posed on $\{t=2\}$ with support contained in $\{|x|<1\}$, then the solution is supported in $\Kcal$. So we restrict ourselves in $\Kcal$.  More details on hyperboloidal foliation can be found in \cite{LM1}.

Return to the discussion on strong coupling. In dimension $3+1$, the strong coupling terms are not critical because of the decay enjoyed by the Klein-Gordon component:
\begin{equation}
|\del v| + |v|\simeq  (s/t)^{3/2}s^{-3/2 + \delta}
\end{equation} 
where $\delta$ measure the increasing rate of its standard energy on hyperboloids. For clarity we call the factor $s^{-3/2+\delta}$  the {\sl principle decay} and the factor $(s/t)^{3/2}$ the {\sl conical decay}. The above decay is integrable with respect to $s$ (in $\Kcal, t^{-1/2} < s/t \leq 1$), so the strong coupling terms always enjoy integrable $L^2$ bounds.

However in $2+1$ dimensional case the role of these terms change dramatically. Even supposing that $v$ enjoys uniform standard energy bounds, one can only obtain the following decay (via Klainerman-Sobolev inequality):
\begin{equation}\label{eq1-13-08-2020}
|\del v| + |v| \simeq (s/t)s^{-1}
\end{equation}
which is {\sl not} integrable. 

What is even worse is that, the strong couplings destroy ``completely'' the conformal invariance of the wave equation in \eqref{eq-main}.  More precisely, consider the term $\del_{\alpha}v\del_{\beta}v$ and regard the conformal energy estimate \eqref{eq8-15-06-2020}, we need to integrate $\|s\del_{\alpha}v\del_{\beta}v\|_{L^2(\Hcal_s)}$ with respect to $s$. This term do not decrease even if we suppose that $v$ enjoys uniform standard energy bound:
$$
\|s\del_{\alpha}v\del_{\beta}v\|_{L^2(\Hcal_s)}\simeq 1.
$$
This leads to at least a $s^{+1}$ increasing rate of conformal energy.


In contrast, in the weak coupling case the quadratic semi-linear terms containing at least one factor of wave component (or its derivatives). These terms  are more friendly because the wave component can be expected to enjoy better decay and $L^2$ bounds due to its conformal energy bounds. However, one can not expect such bounds on Klein-Gordon component. To be more precise, recall the following bounds due to Klainerman-Sobolev inequality:
\begin{equation}
(s/t)|\del u| + (s/t)^{-1}|\delu_a u|\simeq s^{-2}\Ecal_2^N(s,u)^{1/2},
\end{equation}
\begin{equation}
\|(s/t)^2s\del_{\alpha} u\|_{L^2(\Hcal_s)}\simeq \Ecal_2^N(s,u)^{1/2}.
\end{equation}
where $\Ecal_2^N(s,u)$ represents the $N-$order conformal energy defined on hyperboloid $\Hcal_s$ (see in detail in Subsection \ref{subsec-conformal}). For example, let us consider the mixed term $\del_{\alpha}u\del_{\beta}v$ coupled in wave equation and suppose that $v$ enjoys uniform standard energy bound. Then recall \eqref{eq1-13-08-2020} one obtains
$$
s\del_{\alpha}u\del_{\beta}v \simeq (s/t)^{-2}s^{-1} (s/t)^2s|\del_{\alpha} u|
$$
where $ (s/t)^2s|\del_{\alpha} u|$ can be controlled by conformal energy. If we demand null condition on the coefficients of this term, there will be an additional conical decay $(s/t)^2$  which will offset the $(s/t)^{-2}$. Then this bound will be sufficient to recover a slowly increasing conformal energy bound (see for example \cite{M4}). In an other word, in weak coupling case the wave component can be expected to enjoy slowly increasing conformal energy bound, which seems to be impossible in strong coupling case. This is the fundamental difference between the strong and weak coupling cases.  

\subsection{Principle decay v.s. conical decay}
In the above discussion we have applied the term principle decay and conical decay. Due to their importance,
let us give more detailed explanation on their different roles.  In general we write the decay of a term in following form,
\begin{equation}
(s/t)^{\beta}s^{-\alpha}.
\end{equation}
The principle decay determines how fast the function decreases far from the light cone $\del\Kcal = \{r=t-1\}$. It also measures the homogeneity of the solution.  The conical decay describes how much additional decay the solution enjoys near the light-cone. However when $\beta<0$, in order to offset this increasing rate near light-cone, one needs to pay principle decay. Remark that $(s/t)^{-1}\leq s$ in $\Kcal$, then one has
$$
(s/t)^{-\beta}s^{-\alpha}\leq s^{-\alpha + \beta},\quad \beta\geq 0. 
$$
We denote by $-\alpha + [\beta]^-$ the {\sl total decay}, where $[\beta]^-$ denotes the negative part of $\beta$, which is $ - \beta$ when $\beta\leq 0$ and $0$ when $\beta >0$. 
Remark that the $L^2$ norm of the gradient of wave component can not be bounded directly by standard energy on hyperboloid (see in detail \eqref{density-standard}). We need to pay a conical decay $(s/t)$. For example, to bound the $L^2$ norm of $\del_{\alpha} u \del_{\beta} v$, one needs to make the following calculation:
\begin{equation}\label{eq1-22-08-2020}
\|\del_{\alpha}u\del_{\beta}v\|_{L^2(\Hcal_s)}\leq C\|(s/t)\del_{\alpha} u\|_{L^2(\Hcal_s)}\|(s/t)^{-1}\del_{\beta}v\|_{L^{\infty}(\Hcal_s)}
\end{equation}
and demand how much total decay the term $\del_{\beta} v$ enjoy. When this decay is integrable with respect to $s$, one concludes that $\|\del_{\alpha}u\del_{\beta}v\|_{L^2(\Hcal_s)}$ does not blow up the standard energy. The main difficulty comes when this decay is not integarble and it can be classified into two types. One is the lack of conical decay and the other is the lack of principle decay. Both may lead to insufficiency of total decay and even blow-up in finite time. We show some typical examples.  

A first example is the following semi-linear wave equation:
\begin{equation}
\Box u = (\del_t u)^2
\end{equation}
in $\RR^{3+1}$. It is known that all non-zero initial data leads to finite time blow-up (see \cite{John-1981}). This can be observed through \eqref{eq1-22-08-2020}. In fact even if $u$ enjoys uniformly standard energy bound (one can not demand more because for free linear wave equation it is conserved), $(s/t)^{-1}|\del_t u|\sim (s/t)^{-1}(s/t)^{1/2}s^{-3/2}\sim (s/t)^{-1/2}s^{-3/2}$. The principle decay is integrable but the lack of conical decay will offset the principle decay by $s^{1/2}$ and make the total decay non-integrable. 

The second example is the situation in \cite{Kl1} in $\RR^{3+1}$ where one regarded the null quadratic term of wave component $N^{\alpha\beta}\del_{\alpha}u\del_{\beta}u$. With the above observation, we also arrive at \eqref{eq1-22-08-2020}. However this time the null condition supplies a supplementary conical decay $(s/t)^2$ (see in detail \eqref{eq13-10-06-2020}) which makes $\|N^{\alpha\beta}\del_{\alpha}u\del_{\beta}u\|_{L^2(\Hcal_s)}$ integrable. The role of classical null conditions is that they supply additional conical decay.

In some more recent works more delicate techniques are developed.  For example in the situation of \cite{LR2010},  the standard energies only enjoy slowly increasing energy bounds. By Klainerman-Sobolev inequality the wave component satisfies 
\begin{equation}\label{eq1-11-08-2020}
|\del_{\alpha} u|\simeq (s/t)^{-1/2}s^{-3/2+\delta}.
\end{equation}
However in this case the null condition is no longer valid. This is the case of lack of conical decay. The sharp decay estimate in \cite{LR2010} (integration along characteristics) and many other techniques are in fact an exchange of principle decay into conical decay. Remark that the principle decay is sufficient and has a margin (between $-3/2+\delta$ and $-1$), or in another word, the insufficiency of total decay only occurs in the region near light-cone. The sharp decay estimate in \cite{LR2010} scarifies some principle decay to recover the insufficiency of conical decay, and arrive at $|\del u|\simeq (s/t)s^{-1}$, i.e., we lose some principle decay of order $(-1/2 + \delta)$ and recover a conical decay of order  $(3/2)$). Of course, some delicate structures of Einstein equation are also applied for this improvement. Many systems in dimension $3+1$ enjoys the above property because of \eqref{eq1-11-08-2020}. We emphasize that this type of techniques in fact do not demand very much principle decay (in fact $<-1$ is sufficient) and there is still some margin left.

Then we take a look at the Klein-Gordon-Zakharov model system \eqref{eq1-Zakharov} in $\RR^{2+1}$. Clearly this is a strong coupling system. However it is {\sl not} critical if we regard the principle decay. The main observation is that the term $E\Delta w$ coupled in Klein-Gordon equation only contains the Hessian form of the wave component. The gradient $\del w$ does not appear in right-hand-side of the system. Remark that the Hessian form enjoy a faster principle decay (see for example in Proposition \ref{prop1-14-08-2020}). In fact 
$$
|\del\del w|\simeq (s/t)^{-1}s^{-2+\delta} + (s/t)^{-2}|\Box u|
$$
and the last term is quadratic, i.e, it can be expected to enjoy better decay than linear ones. When regarding \eqref{eq1-Zakharov}, $\del\del w\simeq s^{-1}$ is sufficient. So there is a margin between $(-2+\delta)$ and $(-1)$. Of course, to make the argument in \cite{Dong-2020-2} work, there are several non-trivial works in order to overcome the insufficiency of conical decay. Our techniques are also applicable on \eqref{eq1-Zakharov} (though it falls to be a special case of \eqref{eq-main}). In Section \ref{sec-Zakharov} we give a sketch in order to show the difference between critical principle decay and non-critical principle decay.

In the same manner, the system treated in \cite{Stingo-2018} also enjoys the above structure of Hessian form. When restricted to compactly supported initial data, our method is also applicable (see also a generalization in \cite{Dong-2020-2}).

In the case of \eqref{eq-main}, the difficulty comes form the other side: we can show that sufficient conical decay is available but the principle decay is at the critical level, i.e., the uniform standard energy bounds only leads to $s^{-1}$ principle decay for both wave and Klein-Gordon components and this is not integrable. Regarding the interaction $\del_{\alpha}u\del_{\beta}v$ coupled in wave equation, this will lead to non-uniformly-bounded standard energy on wave component. The existing techniques such as in \cite{LR2010} will not be applicable because there is no margin of principle decay, i.e., even if we are far from light-cone, the decay of $\del u$ is still insufficient. The null conditions or the fast decay of Klein-Gordon component (Proposition \ref{prop1-fast-kg}) will not aid because they only affect the conical decay. So we need to develop a series of techniques in a somehow inverse sens, i.e., they permit us to exchange surplus conical decay into principle decay. 

The main idea is to make the critical principle decay $s^{-1}$ sufficient, and when necessary, we can accept a loss on conical decay. Let us consider the energy bounds
\begin{equation}\label{eq3-11-08-2020}
\Ecal_2^N(s,u)^{1/2}\simeq s
\end{equation}
where $\Ecal_2^N(s,u)$ represents the $N$ order conformal energy (see in detail in Subsection \ref{subsec-conformal}). The importance is that to recover this bound, one only needs 
\begin{equation}
\|\Box u\|_{L^2(\Hcal_s)}\leq Cs^{-1}
\end{equation}
and there will not be logarithmic loss. By Klainerman-Sobolev inequality, \eqref{eq3-11-08-2020} leads to the following decay 
\begin{equation}\label{eq2-07-08-2020}
|\del u|\simeq (s/t)^{-1}s^{-1}.
\end{equation}
If we only regard the principle decay, this is sufficient. In an other word, now we have sufficient principle decay, but we have payed $(s/t)^{-1}$. The main techniques to be developed is to show that this prices is acceptable.

In the present article we concentrate only on the techniques of ``paying conical for principle''. The main system \eqref{eq-main} is a model in order to show our mechanism. The choice of this system is made under the following two considerations. First, the system should not be too trivial in order to show the necessity and potential of these techniques. Second, the system should not be too general or complicated such that the main ideas are covered under too much technical details. In our opinion the system \eqref{eq-main} balances well the above two points. We omit all terms that can be treated through existing techniques and preserve all semi-linear terms on $(\del u,\del v, v)$ in wave equation.

The Klein-Gordon-Zakharov model system is not in the form of \eqref{eq-main}, however our techniques are also applicable and the proof is somehow shorter. So we take it as a secondary example. 

In fact the techniques to be developed can be applied on more general systems. For example when there are pure Klein-Gordon terms coupled in Klein-Gordon equation, the normal form method developed by \cite{Shatah85}, \cite{Dfx} and \cite{KS1} are applicable and compatible with these techniques. When considering quasi-linear systems, these techniques can be easily adapted to curved metric.

In a coming work, these techniques will be applied on a special type of strong coupling system deduced form \ref{eq5-08-09-2020}. This will give a preliminary answer to the problem posed in \cite{Ab-2019} on the stability of wave map in $\RR^{2+1}$ case. 

\subsection{Statement of the main results and the structure of this article}
Now we state the main results on \eqref{eq-main} and \eqref{eq1-Zakharov} and then give a brief description on the structure of the present article.
\begin{theorem}\label{main1-thm}
	Consider the Cauchy problem associate to \eqref{eq-main} with the following initial data
	\begin{equation}
	u(2,x) = u_0,\quad \del_t u(2,x) = u_1,\quad v(2,x) = v_0,\quad \del_t v(2,x) = v_1.
	\end{equation}
	Suppose that $u_i$ and $v_i$ are supported in the unit disc $\{|x|<1\}$, sufficiently regular. Then there exists a positive constant $\vep$ such that if for $N\geq 14$ the following bounds hold:
	\begin{equation}
	\|u_0\|_{H^{N+1}(\RR^2)} + \|v_0\|_{H^{N+1}(\RR^2)} + \|u_1\|_{H^N(\RR^2)} + \|v_1\|_{H^N(\RR^2)}\leq \vep,
	\end{equation}
	then the associate local solution extends to time infinity. Furthermore, the following decay bounds hold:
	\begin{equation}
	\aligned
	&|\del u|\leq C\vep t^{-1/2}(1+|t-r|)^{-1/2},\quad |u|\leq C\vep t^{-1/2}(1+|t-r|)^{1/2},
	\\
	&|v| + |\del v|\leq C\vep t^{-1}.
	\endaligned
	\end{equation}
	with $C$ determined by the system and $N$.
\end{theorem}

\begin{theorem}\label{main2-thm}
	Consider the Cauchy problem associate to \eqref{eq1-Zakharov} with the following initial data
	\begin{equation}
	u(2,x) = u_0,\quad \del_t u(2,x) = u_1,\quad E^a(2,x) = E^a_0,\quad \del_t E^a(2,x) = E^a_1.
	\end{equation}
	Suppose that $u_i$ and $E^a_i$ are supported in the unit disc $\{|x|<1\}$, sufficiently regular. Then there exists a positive constant $\vep$ such that if for $N\geq 13$ the following bounds hold:
	\begin{equation}
	\|u_0\|_{H^{N+2}(\RR^2)} + \|E^a_0\|_{H^{N+1}(\RR^2)} + \|u_1\|_{H^{N+1}(\RR^2)} + \|E^a_1\|_{H^N(\RR^2)}\leq \vep,
	\end{equation}
	then the associate local solution extends to time infinity. Furthermore, the following decay bounds hold:
	\begin{equation}
	\aligned
	&|\del u|\leq C\vep t^{-1/2}(1+|t-r|)^{-1/2},\quad |u|\leq C\vep t^{-1/2}(1+|t-r|)^{1/2},
	\\
	&|E^a| + |\del E^a|\leq C\vep t^{-1}.
	\endaligned
	\end{equation}
	with $C$ determined by the system and $N$.
\end{theorem}

This article is composed by two parts. In the first part (from Section \ref{sec-recall} to Section \ref{sec2-08-06-2020}) we develop the necessary estimates and in the second part (from Section \ref{sec-bootstrap-1} to Section \ref{sec-Zakharov}) we prove the global existence of the model systems \eqref{eq-main} and \eqref{eq1-Zakharov}.

For the convenience of the reader we recall some basic results of the hyperboloidal foliation in Section \ref{sec-recall} and sketch some of their proofs in the Appendix. Then as explained in the previous subsections the main task is to surmount the loss of $(s/t)^{-1}$. This is done in two steps. The first step is an estimate on the fundamental solution of the wave equation made in Section \ref{sec1-08-06-2020}. In our mechanism it is a necessary tool for recovering conical decay without paying principle decay on $|u|$. The second step, contained in Section \ref{sec2-08-06-2020}, is an $L^{\infty}$ estimate on $|\del u|$ via integration along hyperbolas. 

The global existence of \eqref{eq-main} is proved by the standard bootstrap argument.  In Section \ref{sec-bootstrap-1} we state the bootstrap argument. In the set of bootstrap bounds there are the energy bounds on wave and Klein-Gordon components as well as the decay bound \eqref{eq17-03-06-2020} on wave component. They will be improved in Section \ref{sec-bootstrap-1}, Section \ref{sec-bootstrap-kg} and Section \ref{sec-bootstrap-decay} respectively. 

In Section \ref{sec-Zakharov} we sketch the global existence of \eqref{eq1-Zakharov}.

\centerline{\bf Acknowledgments}
The present work belongs to a research project ``Global stability of quasilinear wave-Klein-Gordon system in $2 + 1$ space-time dimension'' (11601414), supported by NSFC.

\section{Recall of basic facts in hyperboloidal foliation}
\label{sec-recall}
\subsection{Frames and vector fields}
Let $(t,x)\in \RR^{2+1}$ with $x\in \RR^2$. Denote by $r = \sqrt{|x^1|^2+|x^2|^2}$. We work in the light-cone $\Kcal := \{r < t-1\}\subset \RR^{2+1}$. 

We recall the following nations introduced in \cite{LM1}:
$$
\delu_0:=\del_t,\quad \delu_a := \delb_a = (x^a/t)\del_t + \del_a.
$$
The transition matrices between this frame and the natural frame $\{\del_{\alpha}\}$ are:
\begin{equation}\label{eq semi-frame}
\Phiu_{\alpha}^{\beta} := \left(
\begin{array}{ccc}
1 &0 &0
\\
x^1/t &1 &0
\\
x^2/t &0 &1 
\end{array}
\right),
\quad
\Psiu_{\alpha}^{\beta} := \left(
\begin{array}{ccc}
1 &0 &0
\\
-x^1/t &1 &0
\\
-x^2/t &0 &1 
\end{array}
\right)
\end{equation}
with
$$
\delu_{\alpha} = \Phiu_{\alpha}^{\beta}\del_{\beta},\quad \del_{\alpha} = \Psiu_{\alpha}^{\beta}\delu_{\beta}.
$$

The vector field (derivatives) $\delu_a$ are tangent to the hyperboloid $\Hcal_s = \{(t,x)| t = \sqrt{s^2+r^2}\}$. We call them hyperbolic derivatives.

Let $T = T^{\alpha\beta}\del_{\alpha}\otimes\del_{\beta}$ be a two tensor defined in $\Kcal$ or its subset. Then $T$ can be written with $\{\delu_{\alpha}\}$:
$$
T = \Tu^{\alpha\beta} \delu_{\alpha}\otimes\delu_{\beta}
\quad\text{with}\quad 
\Tu^{\alpha\beta} = T^{\alpha'\beta'}\Psiu_{\alpha'}^{\alpha}\Psiu_{\beta'}^{\beta}.
$$

\subsection{High-order derivatives}
In the region $\Kcal$, we introduce the following Lorentzian boosts:
$$
L_a = x^a\del_t + t\del_a,\quad a = 1, 2
$$
and the following notation of high-order derivatives: let $I,J$ be multi-indices taking values in $\{0,1,2\}$ and $\{1,2\}$ respectively,
$$
I = (i_1,i_2,\cdots, i_m),\quad J = (j_1,j_2,\cdots, j_n).
$$
We define
$$
\del^IL^J = \del_{i_1}\del_{i_2}\cdots \del_{i_m}L_{j_1}L_{j_2}\cdots L_{j_n}
$$
to be an $(m+n)-$order derivative. 

Let $\mathscr{Z}$ be a family of vector fields. $\mathscr{Z} = \{Z_i|i=0,1,\cdots, 6\}$ with
$$
Z_0 = \del_t,\quad  Z_1=\del_1,\quad Z_2 = \del_2,\quad Z_3 = L_1,\quad Z_4=L_2,\quad Z_5 = \delu_1,\quad Z_6 = \delu_2.
$$
A high-order derivative of order $N$ on $\mathscr{Z}$ with milti-index $I = (i_1,i_2,\cdots, i_N)$, $i_j\in\{0,1,\cdots,6\}$ is defined as
$$
Z^I := Z_{i_1}Z_{i_2}\cdots Z_{i_N}.
$$
A high-order derivative $Z^I$ is said to be of type $(a,b,c)$, if it contains at most $a$ partial derivatives, $b$ Lorentzian boosts and $c$ hyperbolic derivatives.

We then introduce the following notation:
$$
\mathcal{I}_{p,k} = \{I| I \text{ is of type }(a,b,0)\text{ with }a+b \leq p, b\leq k \}
$$
so $Z^I$ with $I\in \mathcal{I}_{p,k}$ stands for a high-order derivatives composed by boosts and partial derivative. Its order is smaller or equal to $p$ and it contains at most $k$ boosts.  We define
\begin{equation}\label{eq1 notation}
\aligned
|u|_{p,k} &:= \max_{K\in \mathcal{I}_{p,k}}|Z^K u|,\quad &&|u|_p := \max_{0\leq k\leq p}|u|_{p,k},
\\
|\del u|_{p,k} &:= \max_{\alpha=0,1,2}|\del_{\alpha} u|_{p,k}, &&|\del u|_p := \max_{0\leq k\leq p}|\del u|_{p,k},
\\
|\del^m u|_{p,k} &:= \max_{|I|=m}|\del^I u|_{p,k}, &&|\del^m u|_p := \max_{0\leq k\leq p}|\del^I u|_{p,k},
\\
|\dels u|_{p,k} &:= \max\{|\delu_1 u|_{p,k},|\delu_2u|_{p,k}\}, &&|\dels u|_p := \max_{0\leq k\leq p}|\dels u|_{p,k},
\\
|\del\dels  u|_{p,k} &:=\max_{a,\alpha} \{|\delu_a\del_{\alpha} u|_{p,k},|\del_{\alpha}\delu_a u|_{p,k}\},
&&| \del\dels u|_p :=\max_{0\leq k\leq p}| \del\dels u|_{p,k},
\\
|\dels\dels  u|_{p,k} &:=\max_{a,b} \{|\delu_a\delu_b u|_{p,k}\},
&&| \dels\dels u|_p :=\max_{0\leq k\leq p}| \dels\dels u|_{p,k}.
\endaligned
\end{equation}
These quantities will be applied in order to control varies of high-order derivatives in the following discussion. Our first task is to bound them by energy densities. These results will be stated after the introduction of the standard and conformal energy inequalities in the following two subsections.

\subsection{Standard energy estimate on hyperboloids}
The modifier ``standard'' in the title means that this energy is obtained by the standard multiplier $\del_tu$. Because the coefficient of $\del_t$ is homogeneous of degree zero, this energy is also called {\sl 0-energy}. In the present case we only need energy within Minkowski background metric:
\begin{equation}\label{standard energy}
E_{0,c}(s,u):= \int_{\Hcal_s}e_{0,c}[u]dx
\end{equation}
where the energy density
\begin{equation}\label{density-standard}
\aligned
e_{0,c}[u]:=&|\del_tu|^2+\sum_a|\del_au|^2 + 2(x^a/t)\del_tu\del_au + c^2u^2 
\\
=&\sum_a |\delu_a u|^2 + |(s/t)\del_tu|^2 + c^2u^2
\\
=&|\delu_{\perp}u|^2 + \sum_a|(s/t)\del_a u|^2 + \sum_{a<b}\big|t^{-1}\Omega_{ab}u\big|^2 + c^2u^2
\endaligned
\end{equation}
with $\delu_{\perp} := \del_t + (x^a/t)\del_a$. We denote by $e_0[u] = e_{0,c=0}[u]$.

We also introduce the following high-order energy:
\begin{equation}\label{eq 3' 01-01-2019}
\Ecal_{0,c}^{p,k}(s,u) := \max_{|I|\leq p-k\atop |J|\leq k}E_{0,c}(s,\del^IL^J u),\quad \Ecal_0^{p,k}(s,u) := \max_{|I|\leq p-k\atop |J|\leq k}E_0(s,\del^IL^J u).
\end{equation}
\begin{equation}\label{eq 3 01-01-2019}
\Ecal_{0,c}^N(s,u) := \max_{|I|+|J|\leq N}E_{0,c}(s,\del^IL^J u),\quad \Ecal_0^N(s,u) := \max_{|I|+|J|\leq N}E_0(s,\del^IL^J u)
\end{equation}
and the domain
$$
\Hcal_{[s_0,s_1]} = \Big\{(t,x)\in\Kcal, \sqrt{s_0^2+r^2}\leq t\leq \sqrt{s_1^2+r^2}\Big\}.
$$

Then we recall the standard energy estimate on hyperboloids. The proof can be found in \cite{LM1}.
\begin{proposition}[Standard energy estimate]\label{prop 1 energy}
	We consider the $C^2$ solution $u$ to the following wave / Klein-Gordon equation
	$$
	\Box u + c^2 u = F,
	$$
	in the region $\Hcal_{[s_0,s_1]}$ and vanishes near the conical boundary $\del\Kcal = \{t=r-1\}$. Then the following energy estimate holds:
	\begin{equation}\label{ineq 3 prop 1 energy}
	\aligned
	E_{0,c}(s,u)^{1/2}\leq& E_{0,c}(2,u)^{1/2} + \int_2^s \|F\|_{L^2(\Hcal_\tau)}d\tau.
	\endaligned
	\end{equation}
\end{proposition}

\subsection{Conformal energy estimate on hyperboloids}
\label{subsec-conformal}
In this section we recall the conformal energy estimate and sketch its proof. More detailed discussions can be found in \cite{M4}. For the convenience of discussion, we recall the following hyperbolic parameterization of $\Kcal = \{t>r+1\}$:
$$
\xb^0 = s := \sqrt{t^2 - |x|^2}, \quad \xb^a = x^a.
$$
The related natural frame is
$$
\delb_0 = \delb_s = (s/t)\del_t, \quad \delb_a = \delu_a = (x^a/t)\del_t + \del_a.
$$
Remark that $[\delb_{\alpha},\delb_{\beta}] = 0$. 

We introduce the following conformal energy on hyperboloid $\Hcal_s$:
$$
E_2(s,u) := \int_{\Hcal_s}\Big(\sum_a|s\delu_a u|^2 + s^{-2}|K_2u + su|^2\Big)dx
$$
where $K_2 = s^2\del_t + 2sx^a\delu_a$ is the conformal multiplier. The index $2$ in $K_2$ and $E_2$ is to demonstrate the homogeneity of these objects. In fact the energy controls $s^2\|\delu_a u\|_{L^2(\Hcal_s)}$ and the coefficient of $K_2$ is homogeneous of degree $2$. The same principle is applied on $E_0$ and $K_0 = \del_t$ for standard energy. Then we state the conformal energy estimate:
\begin{proposition}
	Let $u$ be a sufficiently regular function defined in $\Hcal_{[s_0,s_1]}$, vanishes near the conical boundary $\del\Kcal = \{r=t-1\}$. Then the following estimate holds:
	\begin{equation}\label{eq8-15-06-2020}
	\Ec(s_1,u)^{1/2} \leq \Ec(s_0,u)^{1/2} + \int_{s_0}^{s_1}s\|\Box u\|_{L^2(\Hcal_s)}ds.
	\end{equation}
\end{proposition}
\begin{proof}[Sketch of Proof]
	This is a standard energy-type estimate. We apply the multiplier $K_2u + su = s^2\delb_su + 2sx^a\delu_au + su$, write the equation into divergence form and integrate it in the domain $\Kcal_{[s_0,s_1]}$ with Stokes' formula. Here we only give the differential identities. The first is the hyperbolic decomposition of the wave operator:
	$$
	\Box u = s^{-1}\delb_s\left(s\delb_s + 2\xb^a\delb_a\right)u -\sum_a\delb_a\delb_au + \frac{1}{s}\delb_su.
	$$
	Then it is easy to show the following identities:
	$$
	\aligned
	s(s\delb_s + 2\xb^a\delb_a)u\ \Box u=& \frac{1}{2}\left((s\delb_su + 2\xb^a\delb_au)^2 + \sum_a|s\delb_au|^2 \right) + \delb_aw_1^a[u]
	\\
	&+ \delb_su\ (s\delb_s+2\xb^a\delb_s)u - s\sum_a|\delb_au|^2,
	\endaligned
	$$
	$$
	\aligned
	su\ \Box u
	=&\delb_s\left(u(s\delb_s+2\xb^a\delb_a)u + (1/2)u^2\right) + \delb_aw_0^a[u]
	\\
	& - \delb_su\left(\delb_su + 2\xb^a\delb_au\right) + s\sum_a|\delb_au|^2 .
	\endaligned
	$$
	where
	$$
	w_1^a[u] = - s^2\delb_su\delb_au - s\sum_ax^a|\delb_a u|^2  ,\quad w_0^a[u] = -su\delb_a u.
	$$
	Thus we obtain
	\begin{equation}
	\aligned
	&(K_2 + s) u\ \Box u 
	\\
	=& \frac{1}{2}\del_s\Big((s\delb_su + 2\xb^a\delb_au)^2 + \sum_a|s\delb_au|^2 + 2u(s\delb_su + 2\xb^a\delb_au) + u^2\Big) + \delu_aw^a[u].
	\endaligned
	\end{equation}
	Integrate this identity in $\Hcal_{[s_0,s_1]} = \{(s,\xb), s_0\leq s\leq s_1\}$ and apply Stokes' formula, the desired estimate is established. 
\end{proof}
However, the conformal energy does not control directly $\del_t u$. To obtain the bound on this derivative, we recall the following result established in \cite{M4} (for alternative approach, see \cite{Wong-2017}):
\begin{lemma}\label{proposition 1 01-01-2019}
	Let $u$ be a $C^1$ function defined in $\Hcal_{[s_0,s_1]}$ and vanishes near $\del\Kcal$. Then
	\begin{equation}\label{eq 1 14-12-2018}
	\|(s/t)u\|_{L^2(\Hcal_{s_1})}\leq \|(s/t)u\|_{L^2(\Hcal_{s_0})} + C\int_{s_0}^{s_1}s^{-1}\Ec(s,u)^{1/2}ds.
	\end{equation}
\end{lemma}
\begin{proof}
	This relies on the following differential identity:
	\begin{equation}\label{eq 1 01-01-2019}
	\aligned
	(s/t)u\cdot(s/t)\left((s\delb_s + 2\xb^a\delb_a)u + u\right)
	=& \frac{1}{2}s\delb_s\big((s/t)^2u^2\big) + (s/t)u\cdot (x^a/t)s\delb_au
	\\
	&+ \frac{1}{2}\delb_a\big(x^a(s/t)^2u^2\big).
	\endaligned
	\end{equation}
	Integrate this on $\Hcal_s$ (remark that the restriction of $u$ on $\Hcal_s$ is compactly supported), we obtain:
	$$
	\aligned
	&\frac{s}{2}\frac{d}{ds}\int_{\Hcal_s}(s/t)^2u^2\ dx + \int_{\Hcal_s} (s/t)u\cdot (x^a/t)s\delb_au\ dx 
	\\
	&= \int_{\Hcal_s}(s/t)u\cdot(s/t)\left((s\delb_s + 2\xb^a\delb_a)u + u\right)\ dx.
	\endaligned
	$$
	This leads to
	$$
	\aligned
	&\frac{d}{2ds}\|(s/t)u\|_{L^2(\Hcal_s)}^2
	\\
	\leq& Cs^{-1}\|(s/t)u\|_{L^2(\Hcal_s)}\cdot\big(\|(s\delb_s + 2\xb^a\delb_a)u + u\|_{L^2(\Hcal_s)} + \sum_a\|s\delb_au\|_{L^2}\big)
	\\
	\leq& Cs^{-1}\|(s/t)u\|_{L^2(\Hcal_s)}\Ec(s,u)^{1/2}.
	\endaligned
	$$
	Thus
	$$
	\frac{d}{ds}\|(s/t)u\|_{L^2(\Hcal_s)}\leq Cs^{-1}\Ec(s,u)^{1/2}.
	$$
	Then integrate this on time interval $[s_0,s_1]$, the desired result is established.
\end{proof}

For the convenience of discussion, we introduce
$$
F_2(s_0;s,u) =  \|(s/t)u\|_{L^2(\Hcal_{s_0})} + E_2(s,u)^{1/2} + \int_{s_0}^s\tau^{-1}E(\tau,u)^{1/2}d\tau.
$$
Then the following quantities are bounded by $F_2(s,u)$:
\begin{equation}\label{eq21-15-08-2020}
\|(s/t)^2s \del_{\alpha} u\|_{L^2(\Hcal_s)},\quad \|s\delu_a u\|_{L^2(\Hcal_s)},\quad \|(s/t)u\|_{L^2(\Hcal_s)}.
\end{equation}
To see this, one only needs to remark that
$$
(s/t)^2s\del_t u = s^{-1}(s/t)\big(K_2u + su\big) - (s/t)u - 2(x^a/t)s\delu_a u
$$
and the $L^2$ norm of right-hand-side is bounded by $F_2(s_0;s,u)$. For the convenience of discussion, we introduce
$$
\Ecal_2^N(s,u) := \max_{|I|+|J|\leq N}E_2(s,\del^IL^J u), \quad \Fcal_2^N(s_0;s,u):= \max_{|I|+|J|\leq N}F_2(s,\del^IL^J u),
$$
$$
\Ecal_2^{p,k}(s,u) := \max_{|I|+|J|\leq p\atop |J|\leq k}E_2(s,\del^IL^J u), \quad \Fcal_2^{p,k}(s_0;s,u):= \max_{|I|+||J\leq p\atop |J|\leq k}F_2(s,\del^IL^J u). 
$$
\subsection{Calculation with high-order derivatives}\label{sec1-11-06-2020}
\subsubsection{Bounds independent on linear structure of equations}
In this section we briefly recall the estimates on high-order derivatives. These results are established in \cite{LM1}. However, for the convenience of the reader, we will sketch their proof in Appendix \ref{sec-appendix-C}. The constants $C$ appear in this section are determined by the order of derivatives except otherwise specified.

In this subsection, we state the ``functional inequalities'', i.e., these bounds holds for all functions defined in $\Kcal_{[s_0,s_1]}$, sufficiently regular and vanishes near light-cone. In the next Subsection, we state the bounds depending on the structure of D'Alembert operator.

We firstly state the $L^2$ bounds:
\begin{equation}\label{eq1-10-06-2020}
\|(s/t)|\del u|_{p,k}\|_{L^2(\Hcal_s)} + \||\dels u|_{p,k}\|_{L^2(\Hcal_s)} 
+ \|c|u|_{p,k}\|\leq C\Ecal_{0,c}^{p,k}(s,u)^{1/2},
\end{equation}
\begin{equation}\label{eq5-10-06-2020}
\|s|\del\dels u|_{p-1,k-1}\|_{L^2(\Hcal_s)} + \|t|\dels\dels u|_{p-1,k-1}\|_{L^2(\Hcal_s)} \leq C\Ecal_0^{p,k}(s,u)^{1/2},
\end{equation}
\begin{equation}\label{eq2-10-06-2020}
\aligned
\|(s/t)^2s|\del u|_{p,k}\|_{L^2(\Hcal_s)}& + \|s|\dels u|_{p,k}\|_{L^2(\Hcal_s)} 
+  \|(s/t)|u|_{p,k}\|_{L^2(\Hcal_s)}
\\
\leq& C\Fcal_2^{p,k}(s_0;s,u),
\endaligned
\end{equation}
\begin{equation}\label{eq6-10-06-2020}
\aligned
\|(s/t)s^2|\del\dels u|_{p-1,k-1}\|_{L^2(\Hcal_s)}& + \|(s/t)^{-1}s^2|\dels\dels u|_{p-1,k-1}\|_{L^2(\Hcal_s)}
\\
\leq& C\Fcal_2^{p,k}(s_0;s,u).
\endaligned
\end{equation}

Then the $L^{\infty}$ bounds:
\begin{equation}\label{eq3-10-06-2020}
\aligned
\|s|\del u|_{p,k}\|_{L^{\infty}(\Hcal_s)}& + \|t|\dels u|_{p,k}\|_{L^\infty(\Hcal_s)}  + \|ct|u|_{p,k}\|_{L^{\infty}(\Hcal_s)}
\\
\leq& C\Ecal_{0,c}^{p+2,k+2}(s,u)^{1/2}
\endaligned
\end{equation}
\begin{equation}\label{eq7-10-06-2020}
\|st|\del\dels u|_{p-1,k-1}\|_{L^\infty(\Hcal_s)} + \|t^2|\dels\dels u|_{p-1,k-1}\|_{L^{\infty}(\Hcal_s)} \leq C\Ecal_{0,c}^{p+2,k+2}(s,u)^{1/2},
\end{equation}
\begin{equation}\label{eq4-10-06-2020}
\aligned
\|(s/t)s^2|\del u|_{p,k}\|_{L^\infty(\Hcal_s)} + \|(s/t)^{-1}s^2|\dels u|_{p,k}\|_{L^\infty(\Hcal_s)} 
&+ \|s|u|_{p,k}\|_{L^\infty(\Hcal_s)}
\\
\leq& C\Fcal_2^{p+2,k+2}(s_0;s,u),
\endaligned
\end{equation}
\begin{equation}\label{eq8-10-06-2020}
\aligned
\|s^3|\del\dels u|_{p-1,k-1}\|_{L^{\infty}(\Hcal_s)}& + \|(s/t)^{-2}s^3|\dels\dels u|_{p-1,k-1}\|_{L^{\infty}(\Hcal_s)} 
\\
\leq& C\Fcal_2^{p+2,k+2}(s_0;s,u).
\endaligned
\end{equation}

The idea is as following . First, we bound the quantities listed \eqref{eq1 notation} by the following ``standard'' high-order derivatives $\del_{\alpha}\del^IL^J u$:
\begin{equation}\label{eq11-10-06-2020}
\aligned
|\del u|_{p,k}\leq& C\sum_{|I|+|J|\leq p\atop |J|\leq k,\alpha}|\del_{\alpha}\del^IL^Ju|,
\\
|(s/t)\del u|_{p,k}\leq& C(s/t)\sum_{|I|+|J|\leq p\atop |J|\leq k,\alpha}|\del_{\alpha}\del^IL^Ju|,
\endaligned
\end{equation} 
\begin{equation}\label{eq3 notation}
|\dels u|_{p,k}\leq C\sum_{|I|\leq p-k,a\atop |J|\leq k}|\delu_a\del^IL^Ju|
+ Ct^{-1}\sum_{ |J|\leq k,\alpha\atop 0\leq |I|\leq p-k-1}|\del_{\alpha}\del^IL^Ju|\leq Ct^{-1}|u|_{p+1,k+1},
\end{equation}
\begin{equation}\label{eq4 notation}
\aligned
|\dels\dels u|_{p,k}\leq& Ct^{-1}\sum_{|I|\leq p-k, a\atop |J|\leq k+1}|\delu_a\del^IL^Ju|
+ Ct^{-2}\sum_{ |J|\leq k+1,\alpha\atop 0\leq |I|\leq p-k-1}|\del_{\alpha}\del^IL^Ju|
\\
\leq& Ct^{-2}|u|_{p+2,k+2},
\endaligned
\end{equation}
\begin{equation}\label{eq1 lem5 notation}
|\del\dels u|_{p,k}\leq Ct^{-1}|\del u|_{p+1,k+1}.
\end{equation}

These standard forms are easily controlled by energy densities (with suitable weights, for conformal energy, apply the bounds on \eqref{eq21-15-08-2020}). Then \eqref{eq1-10-06-2020} -- \eqref{eq6-10-06-2020} are established.

For \eqref{eq3-10-06-2020} - \eqref{eq8-10-06-2020}, we need the following Klainermain-Sobolev type inequality on hyperboloids:
\begin{equation}
t|u(t,x)| \leq C\sum_{|I|+|J|\leq 2}\|\del^IL^J u\|_{L^2(\Hcal_s)},\quad (t,x)\in\Hcal_s
\end{equation} 
where $u$ is a $C^2$ function defined in $\Kcal_{[s_0,s_1]}$ vanishes near light-cone $\del\Kcal = \{r=t-1\}$. Then we regard for example for $K \in \mathcal{I}_{p,k}$:
$$
\sum_{|I|+|J|\leq 2}|\del^IL^J \big((s/t)Z^K \del_{\alpha}u\big)|\leq C\sum_{|I'|+|J'|\leq p+2\atop |J'|\leq k+2,\beta}|(s/t)\del_{\beta}\del^{I'}L^{J'}u|.
$$
This leads to $s|Z^K\del_{\alpha}u(t,x)|\leq C\Ecal_0^{p+2,k+2}(s,u)^{1/2}$. Here we have applied the relation
$$
|\del^IL^J(s/t)|\leq C(s/t)
$$
in $\Kcal$. This can be proved easily by induction. 

We also need the following bound on products and null quadratic forms in $\Hcal_{[s_0,s_1]}$. Firstly,
\begin{equation}\label{eq12-10-06-2020}
|AB|_{p,k}\leq C|A|_{p,k}|B|_{p_1,k_1} + C|A|_{p_1,k_1}|B|_{p,k} 
\end{equation}
where $p_1 = [p/2], k_1 = [k/2]$, $A,B$ sufficiently regular in $\Kcal_{[s_0,s_1]}$ and $C$ a constant determined by $p$. Furthermore, let $A$ be a (constant coefficient) quadratic null form, i.e.,
$$
A^{\alpha\beta}\xi_{\alpha}\xi_{\beta} = 0,\quad \forall \xi_0^2 - \xi_1^2 - \xi_2^2 = 0.
$$
Then
\begin{equation}\label{eq13-10-06-2020}
\aligned
|A^{\alpha\beta}\del_{\alpha}u\del_{\beta}v|_{p,k}\leq& C|A|(s/t)^2|\del u|_{p_1,k_1}|\del v|_{p,k} + C(s/t)^2|A||\del u|_{p,k}|\del v|_{p_1,k_1} 
\\
&+ C|A||\dels u|_{p_1,k_1}|\del u|_{p,k} + C|A||\dels u|_{p,k}|\del v|_{p_1,k_1}
\\
&+C|A||\del u|_{p_1,k_1}|\dels v|_{p,k} + C|A||\del u|_{p,k}|\dels v|_{p_1,k_1}
\endaligned
\end{equation}
where $|A| = \max_{\alpha,\beta}|A^{\alpha\beta}|$. The proof is sketched in Appendix \ref{subsec-null}.
\subsubsection{Bounds depending on the linear structure of equations}
We briefly recall two estimates depending on the semi-hypoerboloidal decomposition of the D'Alembert operator.
\begin{proposition}\label{prop1-14-08-2020}
	Let $u$ be a function defined in $\Hcal_{[s_0,s_1]}$, sufficiently regular. Suppose that $|I|+|J|\leq p$ and $|J|\leq k$. Then 
	\begin{equation}\label{eq1 lem Hessian-flat-zero}
	(s/t)^2|\del_\alpha\del_\beta \del^IL^J u| \leq C|\Box u|_{p,k} + Ct^{-1}|\del u|_{p+1,k+1}.
	\end{equation}
	\begin{equation}\label{eq2 lem Hessian-flat-zero}
	(s/t)^2|\del\del  u|_{p,k} \leq C|\Box u|_{p,k} + Ct^{-1}|\del u|_{p+1,k+1}.
	\end{equation}
\end{proposition}

\begin{proposition}\label{prop1-fast-kg}
	Let $v$ be a regular solution to
	\begin{equation}\label{eq1 prop fast-KG}
	\Box v + c^2 v = f.
	\end{equation}
	Then 
	\begin{equation}\label{eq2 prop fast-KG}
	c^2|v|_{p,k}\leq C(s/t)^2|\del v|_{p+1,k+1} + C|f|_{p,k}.
	\end{equation}
\end{proposition}
Both are based on the following semi-hyperboloidal decomposition:
\begin{equation}\label{eq2-11-06-2020}
\Box = (s/t)^2\del_t\del_t + t^{-1}\underbrace{\left((2x^a/t)\del_tL_a - \sum_a\delu_aL_a - (x^a/t)\delu_a + (2+(r/t)^2)\del_t\right)}_{A_m[u]}
\end{equation}
and $|A_m[u]|\leq C|\del u|_{1,1}$. Furthermore \eqref{eq1 prop fast-KG} leads to
$$
c^2 v = - (s/t)^2\del_t\del_t v - t^{-1}A_m[u] + f
$$
which leads to \eqref{eq2 prop fast-KG}. Remark that in $\Kcal$, $t^{-1}\leq C(s/t)^2$.

\section{Decay bounds of wave equation based on Poisson's formula}\label{sec1-08-06-2020}
As explained in Introduction, the main purpose of this estimate is to recover sufficient conical decay on $|u|_{p,k}$ without paying principle decay. Observe that in the wave equation of \eqref{eq-main}, the strong coupling terms enjoy the decay $(s/t)^ks^{-2}$ ($k$ can be large). It seems to be difficult to improve this $(-2)$ principle decay rate (recalling the $s^{-1}$ principle decay of free-linear Klein-Gordon equation). By homogeneity of the Poisson's formula, one can only expect a zero order principle decay on $|u|$. The following estimate tell us that we can do a bit better, in fact a conical decay $(s/t)$ can be preserved when the source enjoys sufficient conical decay.  
\subsection{The estimate}
In this section we will establish the following bound:
\begin{proposition}\label{prop1-01-06-2020}
	Let $u$ be a sufficiently regular solution to the following Cauchy Problem
	\begin{equation}\label{eq7-01-06-2020}
	\Box u = F,\quad 
	u(t_0,x) = u_0, \quad \del_t u(t_0,x) = u_1, \quad t_0\geq 2
	\end{equation}
	with $u_0,u_1$ being compactly supported in $\{|x| < t_0-1\}$ and sufficiently regular, $t_0\geq 2$. Suppose that
	\begin{equation}\label{eq8-01-06-2020}
	|u_0(x)| + |\del_xu_0(x)| + |u_1|\leq C_I.
	\end{equation}
	$F$ is a sufficiently regular function satisfying
	\begin{equation}\label{eq1-30-05-2020}
	F\leq C_F\mathbbm{1}_{\{|x|\leq t-1\}}t^{-3/2-\mu}(t-r)^{-1/2+\nu}, \quad 0< \mu\leq \nu\leq 1/2.
	\end{equation}
	Then for $ 9t/10\leq r\leq t-1$, 
	\begin{equation}\label{eq9-01-06-2020}
	|u(t,x)|\leq \bar{C}C_F\mu^{-1}t^{\nu-\mu}(s/t) + \bar{C}C_I s^{-1}
	\end{equation}
	where $\bar{C}$ is a universal constant.
\end{proposition}
One needs to decompose $u$ as
$$
\Box u_F = F,\quad u_F(t_0,x) = \del_t u_F(t_0,x) = 0
$$
and
$$
\Box u_I = 0,\quad u_I(t_0,x) = u_0,\ \del_t u_I(t_0,x) = u_1.
$$
Obviously, $u = u_F+u_I$ by uniquenss theory. The bound on $u_I$ is established in Lemma \ref{lem1-13-06-2020} and the bound on $u_F$ is given by Lemma \ref{lem3 Poisson}.

\begin{remark}
	The decay on $F$ can be written into the following equivalent form
	$$
	|F|\leq C_F\mathbbm{1}_{\{|x|\leq t-1\}}(s/t)^{1+\mu+\nu}s^{-2-\mu+\nu}.
	$$
\end{remark}

\begin{remark}
	One may compare this result with the Lemma 3.2 and Lemma 3.3 in \cite{Hou-2020}. Here we regard an endpoint case with no loss of principle decay (or equivalently, the homogeneity. Of course here we have the restriction on the supports of initial data and source term). This is because in our case the principle decay has no margin compared with the pure wave case, where the conformal invariance will lead to stronger principle decay.
\end{remark}

\begin{lemma}\label{lem1-13-06-2020}
	Let $u$ be the $C^2$ solution to the following Cauchy problem of free-linear wave equation:
	\begin{equation}
	\Box u = 0,\quad u(t_0,x) = u_0, \quad \del_tu(t_0,x) = u_1, \quad t_0\geq 2
	\end{equation}
	with $u_0,u_1$ sufficiently regular and compactly supported in $\{|x|<t_0-1\}$. Suppose that
	$$
	|u_0(x)| + |\del_xu(x)| + |u_1(x)|\leq C_I.
	$$
	Then for $(t,x)\in\Kcal = \{r<t-1\}$,
	\begin{equation}\label{eq1-17-08-2020}
	|u(t,x)|\leq \bar{C}C_I t_0^2 s^{-1}.
	\end{equation}
	Here $\bar{C}$ is a universal constant.
\end{lemma}
\begin{proof}[Proof of Lemma \ref{lem1-13-06-2020}]
	We need to recall the following Poisson's formula:
	\begin{equation}
	\aligned
	u_I(t,x) =& \frac{1}{2\pi}\int_{|y-x|<t-t_0}\frac{u_1(y)dy}{\sqrt{(t-t_0)^2-|y-x|^2}} 
	\\
	&+ \frac{1}{2\pi}\del_t \Big(\int_{|y-x|<t-t_0}\frac{u_0(y)dy}{\sqrt{(t-t_0)^2-|y-x|^2}} \Big)
	\\
	=:& I_1(t,x) + I_2(t,x).
	\endaligned
	\end{equation}
	The remaining task is to bound the two integrals in right-hand-side. 
	
	For $I_1$, remark that due to the support of $u_1$, we only need to consider the case when $|y|\leq t_0-1$. Then,
	\\
	- when $t-r\geq 4t_0$, $t-r-2t_0\geq \frac{1}{2}(t-r)$. Remark that $t+r\geq t\geq t_0\geq 2$ leads to $t+r - 1 \geq \frac{1}{2}(t+r)$. Furthermore, $|y-x|\leq t_0+r-1$. Then
	$$
	\aligned
	(t-t_0)^2 - |y-x|^2\geq& (t-t_0)^2 - (r+t_0-1)^2 
	\\
	=& (t-r-2t_0+1)(t + r-1)\geq \frac{1}{4}(t-r)(t+r)
	\\
	=& \frac{1}{4}(t^2-r^2).
	\endaligned
	$$ 
	So 
	$$
	\aligned
	|I_1(t,x)|
	\leq \bar{C}C_I\int_{|y-x| < t-t_0 }\mathbbm{1}_{\{|y|< t_0-1\}}(t^2-r^2)^{-1/2}dy
	\leq \bar{C}C_It_0^2s^{-1}
	\endaligned
	$$
	- when $1< t-r\leq 4t_0$. In this case, we need to introduce the following parametrization on the plane $\{t=t_0\}$. We only worry about the bound where $t$ is large, so we firstly suppose that $t\geq 5t_0$. In this case $r>t-4t_0  \geq t_0 \geq 1$. Without loss of generality, we take $x = (r,0)$ on the plane $\{t=t_0\}$. Let $\theta\in(-\pi,\pi)$ be the angle from the vector $(1,0)$ to the vector $(y-x)$ and $\rho = |y-x|$. Then $dy = \rho d\theta d\rho$.

	Then
	$$
	|I_1(t,x)|\leq \bar{C}C_I\int_{\RR^2}\frac{\mathbbm{1}_{\{r-t_0+1<\rho<t-t_0\}}\mathbbm{1}_{\{|\pi-\theta|<\theta_0\}}}{\sqrt{(t-t_0)^2 - \rho^2}}dy
	$$ 
	here remark that the set $\{r-1 < \rho<t-t_0\}\cap\{|\pi-\theta|<\theta_0\}$ covers the set $\{|y-x|<t-t_0\}\cap \{|y|<t_0-1\}$ where $\theta_0\in (0,\pi)$ such that $\sin\theta_0 = (t_0-1)/r$. Remark that $0\leq \theta_0\leq \pi/2$, $\theta_0 \leq C\sin\theta_0\leq C\frac{t_0-1}{r}$. Recall that
	$$
	t-r\leq 4t_0\, \Rightarrow\, r\geq t-4t_0.  
	$$
	On the other hand, $t\geq 5t_0\,\Rightarrow\, t_0\leq t/5\Rightarrow t-4t_0\geq t/5$. So we obtain:
	\begin{equation}\label{eq1-10-09-2020}
	r\geq t/5
	\end{equation}
	and thus
	\begin{equation}\label{eq2-10-09-2020}
	\theta_0\leq \bar{C}\sin\theta_0\leq \bar{C}t_0/t.
	\end{equation}
	Then
	$$
	\aligned
	|I_1(t,x)|\leq& \bar{C}C_I\int_{r-1<\rho<t-t_0}\int_{\pi - \theta_0}^{\pi}\frac{\rho d\theta d\rho}{\sqrt{(t-t_0)^2 - \rho^2}}
	\\
	= & \bar{C}C_I \theta_0 \int_{r-1}^{t-t_0}\big((t-t_0)^2 - \rho^2\big)^{-1/2}\rho d\rho
	\\
	=& \bar{C}C_I \theta_0(t-t_0) \int_{\big(\frac{r-1}{t-t_0}\big)^2}^{1}\big(1 - w\big)^{-1/2}d w
	\endaligned
	$$
	with $w = \Big(\frac{\rho}{t-t_0}\Big)^2$. Then
	$$
	|I_1(t,x)|\leq \bar{C}C_I\theta_0(t-t_0)\Big(1-\Big(\frac{r-1}{t-t_0}\Big)^2\Big)^{1/2}\leq \bar{C}C_It_0^{1/2}\theta_0(t-t_0)^{1/2}.
	$$
	This is because
	$$
	\aligned
	1 - \Big(\frac{r-1}{t-t_0}\Big)^2 =& \frac{(t-t_0- (r-1))(t-t_0+r-1)}{(t-t_0)^2}
	\leq \bar{C}\frac{t-t_0-(r-1)}{t-t_0}
	\\
	\leq& \bar{C}\frac{3t_0+1}{t-t_0}
	\endaligned
	$$
	where for the first inequality,
	$$
	t-r\geq 4t_0 \Rightarrow t-t_0\geq r+3t_0\geq r-1 \Rightarrow \frac{t-t_0 + r-1}{t-t_0}\leq 2.
	$$
	So we conclude that when $t-r\leq 4t_0$ and $t\geq 5t_0$, 
	\begin{equation}\label{eq1-13-06-2020}
	|I_1(t,x)|\leq \bar{C}C_It_0^{3/2}t^{-1/2}.
	\end{equation}
	Now remark that $1\leq (t-r)^{1/2}\leq 2t_0^{1/2}$, \eqref{eq1-13-06-2020} leads to
	\begin{equation}\label{eq2-13-06-2020}
	|I_1(t,x)|\leq \bar{C}C_It_0^2t^{-1/2}(t-r)^{1/2}(t-r)^{-1/2}\leq \bar{C}C_It_0^2s^{-1}
	\end{equation}
	when $t-r\leq 4t_0$ and $t\geq 5t_0$.
	
	When $t_0\leq t\leq 5t_0$, remark that on the region $\{t_0\leq t\leq 5t_0\}$, $\sqrt{2t_0-1}\leq s \leq 5t_0$. So we conclude that
	\begin{equation}\label{eq3-13-06-2020}
	|u(t,x)|\leq \bar{C}C_It_0^2s^{-1}.
	\end{equation}
	
	For $I_2$ the proof is similar. Remark that because $u_0$ is compactly supported and sufficiently regular,
	$$
	\aligned
	&I_2(t,x) 
	\\
	=& \frac{1}{2\pi}\del_t\Big((t-t_0)\int_{|z|<1}\frac{u_0(x+(t-t_0)z)}{\sqrt{1-|z|^2}}dz\Big)
	\\
	=& \frac{1}{2\pi}\int_{|z|<1}\frac{u_0(x+(t-t_0)z)}{\sqrt{1-|z|^2}}dz 
	+ \frac{1}{2\pi}(t-t_0)\int_{|z|<1}\frac{z^a\del_au(x+(t-t_0)z)}{\sqrt{1-|z|^2}}dz
	\endaligned
	$$
	where $z = \frac{y-x}{t-t_0}$. Then we change back the integral variable and find:
	$$
	\aligned
	I_2(t,x) =& \frac{1}{2\pi(t-t_0)}\int_{|y-x|<t-t_0}\frac{u_0(y)dy}{\sqrt{(t-t_0)^2-|y-x|^2}}
	\\
	&+ \frac{1}{2\pi}\int_{|y-x|<t-t_0}\frac{\frac{y^a-x^a}{t-t_0}\del_au(y)dy}{\sqrt{(t-t_0)^2 - |y-x|^2}}.
	\endaligned
	$$
	
	Remark that both $u_0$ and $\frac{y^a-x^a}{t-t_0}\del_a u_0$ are still supported in $|y|<t_0-1$ and, when $t\geq 2t_0$, they are bounded by $\bar{C}C_I$. Then the bound on $I_1$ is applicable. Then the desired bound is established.
	
	One may worry about the first term with the singular denominator $t-t_0$ when $t\rightarrow t_0^+$. In fact we only need to bound the solution for large $t$. When $t_0\leq t\leq 2t_0$, by local theory the solution is bounded then it can be bounded by  \eqref{eq1-17-08-2020}
\end{proof}

\begin{lemma}\label{lem3 Poisson}
	Let $u$ be a $C^2$ solution to the Cauchy problem 
	$$
	\Box u = F,\quad u(t_0,x)  = \del_t u(t_0,x ) = 0
	$$
	with $t_0\geq 2$ and $F$ satisfying \eqref{eq1-30-05-2020}.
	Then the following bound holds for $9t/10\leq |x|<t-1$:
	\begin{equation}\label{eq1 lem3 Poisson}
	|u(t,x)|\leq \bar{C}C_F\mu^{-1}t^{\nu-\mu}(s/t)
	\end{equation}
	where $\bar{C}$ is a universal constant.
\end{lemma}
\begin{proof}
	The proof relies on the following Poisson's formula:
	\begin{equation}\label{eq1-04-06-2020}
	u(t,x) = \frac{1}{2\pi}\int_{t_0}^t\int_{|y-x|<t-\tau}\frac{F(\tau,y)}{\sqrt{(t-\tau)^2 - |y-x|^2}}dy\ d\tau.
	\end{equation}
	Taking \eqref{eq1-30-05-2020}, we obtain:
	$$
	|u(t,x)|\leq \bar{C}C_F\int_{t_0}^t\int_{|y-x|<t-\tau\atop |y|<\tau-1}
	\frac{\tau^{-3/2-\mu}(\tau - |y|)^{-1/2+\nu}}{\sqrt{(t-\tau)^2 - |y-x|^2}}dy\ d\tau.
	$$
	Denote by 
	$$
	\lambda = \tau/t,\quad z = y/t, 
	$$
	the above inequality is written as
	\begin{equation}
	|u(t,x)|\leq \bar{C}C_Ft^{\nu-\mu} \int_{t_0/t}^1 \lambda^{-3/2-\mu}
	\int_{|z-x/t|\leq 1-\lambda\atop |z|<\lambda-t^{-1}} \frac{(\lambda - |z|)^{-1/2+\nu}dz \,d\lambda}{\big((1-\lambda)^2 - |z-x/t|^2\big)^{1/2}}.
	\end{equation}
	We denote by 
	$$
	I_{t,x}(\lambda) := \lambda^{-3/2-\mu}\int_{|z-x/t|\leq 1-\lambda\atop |z|<\lambda-t^{-1}}\!\!\!\!\!\!\!\!
	(\lambda - |z|)^{-1/2+\nu}\big((1-\lambda)^2 - |z-x/t|^2\big)^{-1/2}dz \,d\lambda
	$$
	and following the technical Lemma \ref{lem1-30-05-2020}, the desired result is obtained.
\end{proof}

\begin{lemma}\label{lem1-30-05-2020}
	Following the above definition, for $9t/10\leq |x| < t-1$, 
	\begin{equation}
	\int_{t_0/t}^1I_{t,x}(\lambda)d\lambda \leq \bar{C}C_F \mu^{-1} (s/t)^{1+2(\nu-\mu)}.
	\end{equation}
\end{lemma}
The proof of this technical result is included in the next subsection. 

\subsection{Proof of Lemma \ref{lem1-30-05-2020}}\label{proof-Poisson-f}
\subsubsection{Parametrization}
Because $t_0\geq 2$ and $|x|\geq 9t/10\geq 9t_0/10 \geq 9/5 >1$, with out lose of generality, we take $x = (r,0)\in \RR^2$, i.e., $x/t = (r/t,0)$. Then we consider two discs:
$$
B_O(\lambda) = \{|z|<\lambda-t^{-1}\},\quad B_A(\lambda) = \{|z-x/t|<1-\lambda\}
$$ 
and let $O = (0,0)$, $A = (r/t,0)$.
The integration is made on $B_O(\lambda)\cap B_A(\lambda)$. To calculate $I_{t,x}$, we need the following parametrization:
$$
\aligned
\rho =& |z|,
\\
\theta =& \text{the angle from (1,0) to $z$},\quad \theta\in (-\pi,\pi). 
\endaligned
$$
Then by elementary trigonometry,
\begin{equation}\label{eq4-05-30-2020}
|z-x/t|^2 = \rho^2 + (r/t)^2 - 2(r/t)\rho \cos\theta.
\end{equation}
Recall the expression of $I_{t,x}$: 
\begin{equation}
I_{t,x}(\lambda) = \lambda^{-3/2-\mu}\int_{|\rho|<\lambda-t^{-1}\atop |z-x/t|<1-\lambda}
\frac{(\lambda - \rho)^{-1/2+\nu}\rho d\rho d\theta}{(1-\lambda + |z-x/t|)^{1/2}(\big|1-\lambda - |z-x/t|\big|)^{1/2}}.
\end{equation}
Then,  in order to decide the range of the parameters $\rho,\theta$, we need to discuss the relative position of $B_O(\lambda)$ and $B_A(\lambda)$.  When $\lambda$ varies in $[t_0/t,1]$, there are four cases:
\\
Case I: $B_O(\lambda)\subset B_A(\lambda)\ \Leftrightarrow\ t/t_0\leq \lambda \leq \frac{t-r+1}{2t}$,
\\
Case II: $B_O(\lambda)\not\subset  B_A(\lambda), O\in B_A(\lambda) \ \Leftrightarrow\ \frac{t-r+1}{2t}<\lambda < \frac{t-r}{t}$,
\\
Case III: $B_O(\lambda)\not\subset  B_A(\lambda), O\not\in B_A(\lambda)\ \Leftrightarrow\ \frac{t-r}{t}\leq \lambda <\frac{t+r+1}{2t}$,
\\
Case IV: $B_A(\lambda)\subset B_O(\lambda)\ \Leftrightarrow\ \frac{t+r+1}{2t}\leq \lambda \leq 1$.

\subsubsection{Case I: $B_O(\lambda)\subset B_A(\lambda)$}
Recalling $t_0/t\leq \lambda \leq \frac{t-r+1}{2t}$, one has
\begin{equation}\label{eq2-30-05-2020}
t_0/t\leq \lambda\leq \frac{t-r+1}{2t} \leq \frac{1}{2} - \frac{r-1}{2t} < \frac{1}{2}.
\end{equation}
In this case the integral is made on $B_O$, thus $\theta\in (-\pi,\pi)$, $0\leq \rho\leq \lambda-t^{-1}$, and 
$$
\aligned
\big|1-\lambda - |z-x/t|\big|\geq& 1-\lambda  - (\lambda - t^{-1} + r/t) = \frac{t-r+1}{t} - 2\lambda,
\\
1-\lambda + |z-x/t| \geq& 1/2.
\endaligned
$$
Then
$$
\aligned
I_{t,x}(\lambda) \leq& \bar{C}\lambda^{-3/2-\mu}\Big(\frac{t-r+1}{2t}-\lambda\Big)^{-1/2}
\int_0^{\lambda-t^{-1}}\rho(\lambda - \rho)^{-1/2 + \nu}d\rho
\\
\leq& \bar{C}\lambda^{-1-\mu+\nu}(\lambda-t^{-1})\Big(\frac{t-r+1}{2t}-\lambda\Big)^{-1/2}
\\
\leq& \bar{C}\lambda^{\nu-\mu}\Big(\frac{t-r+1}{2t}-\lambda\Big)^{-1/2}.
\endaligned
$$
Then if we integrate $I_{t,x}$ on $\lambda \in [t_0/t, \frac{t-r+1}{2t}]$, we obtain:
$$
\aligned
\int_{t_0/t}^{\frac{t-r+1}{2t}}I_{t,x}(\lambda)d\lambda
\leq& \bar{C}C_F\int_{t_0/t}^{\frac{t-r+1}{2t}}\lambda^{\nu-\mu}\Big(\frac{t-r+1}{2t}-\lambda\Big)^{-1/2}d\lambda
\\
\leq& \bar{C}C_F\Big(\frac{t-r}{t}\Big)^{\nu-\mu}\int_{t_0/t}^{\frac{t-r+1}{2t}}\Big(\frac{t-r+1}{2t}-\lambda\Big)^{-1/2}d\lambda
\\
\leq& \bar{C}C_F\Big(\frac{t-r}{t}\Big)^{1/2 + \nu-\mu}.
\endaligned
$$
Here for the second inequality we have remarked that $\nu\geq \mu$ and $t_0/t\leq \lambda \leq 1$. Recall that in $\Kcal$, $ \frac{t-r}{t}\leq (s/t)^2$, we conclude that
\begin{equation}\label{eq6-31-05-2020}
\int_{t_0/t}^{\frac{t-r+1}{2t}}I_{t,x}(\lambda)d\lambda \leq \bar{C}C_F(s/t)^{1+2(\mu-\nu)}.
\end{equation}

\subsubsection{Case II: $B_O(\lambda)\not\subset B_A(\lambda), O\in B_A(\lambda)$}
In this case $\frac{t-r+1}{2t} < \lambda < 1-r/t$. Recall the expression of $I_{t,x}$, one has
$$
I_{t,x}(\lambda) = \lambda^{-3/2-\mu}\int_0^{\lambda-t^{-1}}\!\!\!\!\! \int_{-\theta_0(\rho,\lambda)}^{\theta_0(\rho,\lambda)}\frac{(\lambda-\rho)^{-1/2+\nu}\rho d\rho d\theta}{\big|(1-\lambda) - |z-x/t|\big|^{1/2}\big|(1-\lambda)+|z-x/t|\big|^{1/2}}
$$
where the bound $\theta_0 = \theta_0(\rho,\lambda)$ is determined by $\rho,\lambda$:
\\
- when $0\leq \rho < 1-\lambda-r/t$, $\forall \theta\in(-\pi,\pi)$, $|z - x/t| < 1-\lambda$, $\theta_0 = \pi$,
\\
- when $1-\lambda -r/t\leq \rho <\lambda - t^{-1}$,  $\theta_0$ is determined by the following trigonometrical equation: 
\begin{equation}\label{eq2-31-05-2020}
\rho^2 + (r/t)^2 - 2\rho(r/t)\cos \theta_0 = (1-\lambda)^2
\end{equation}
which means when we take $|\theta| = \theta_0$, $|z-x/t| = 1-\lambda$ with $|z| = \rho$. So we make the following decomposition of $I_{t,x}$:
$$
I_{t,x}(\lambda) = I_{t,x}^1(\lambda) + I_{t,x}^2(\lambda)
$$
with
$$
I_{t,x}^1(\lambda) := \lambda^{-3/2-\mu}\int_{|z|<\lambda-t^{-1}\atop |z-x/t|<1-\lambda}
\frac{(\lambda -\rho)^{-1/2+\nu}\mathbbm{1}_{\{0\leq \rho\leq 1-\lambda -r/t\}}\rho d\rho d\theta}{\big|(1-\lambda) - |z-x/t|\big|^{1/2}\big|(1-\lambda)+|z-x/t|\big|^{1/2}},
$$
$$
I_{t,x}^2(\lambda) := \lambda^{-3/2-\mu}\int_{|z|<\lambda-t^{-1}\atop |z-x/t|<1-\lambda}
\frac{(\lambda -\rho)^{-1/2+\nu} \mathbbm{1}_{\{1-\lambda -r/t\leq \rho\leq \lambda-t^{-1}\}}\rho d\rho d\theta}{\big|(1-\lambda) - |z-x/t|\big|^{1/2}\big|(1-\lambda)+|z-x/t|\big|^{1/2}}.
$$

$I_{t,x}^1$ is much easier than $I_{t,x}^2$, we firstly regard this one.
$$
\aligned
I_{t,x}^1 =& \lambda^{-3/2-\mu}\int_0^{1-\lambda-r/t}\int_{-\pi}^{\pi}\frac{\rho(\lambda-\rho)^{-1/2+\nu}d\theta\, d\rho}{(1-\lambda-|z-x/t|)^{1/2}(1-\lambda +|z-x/t|)^{1/2}}.
\endaligned
$$
Then remark that
$$
\aligned
&(1-\lambda-|z-x/t|)^{-1/2}\leq (1-\lambda -r/t-\rho)^{-1/2},
\\
&(1-\lambda+|z-x/t|)^{-1/2}\leq (1-\lambda)^{-1/2}\leq (r/t)^{-1/2}\leq \bar{C},
\\
&(\lambda-\rho)^{-1/2+\nu}\leq \big(\lambda - (1-r/t-\lambda)\big)^{-1/2+\nu}\leq \bar{C}\Big(\lambda  - \frac{t-r}{2t}\Big)^{-1/2+\nu}.
\endaligned
$$
Then
$$
\aligned
I_{t,x}^1 \leq& \bar{C}\lambda^{-3/2-\mu}(1-r/t-\lambda )\Big(\lambda - \frac{t-r}{2t}\Big)^{-1/2+\nu} \int_0^{1-r/t-\lambda}\!\!\!\!\!\!\!\!\!(1-\lambda-r/t-\rho)^{-1/2}d\rho.
\endaligned
$$
So we conclude by
\begin{equation}\label{eq3-30-05-2020}
I_{t,x}^1(\lambda)\leq  \bar{C}\lambda^{-3/2-\mu}\Big(\lambda - \frac{t-r}{2t}\Big)^{-1/2+\nu}(1-r/t-\lambda)^{3/2}.
\end{equation}

$I_{t,x}^2$ is more complicated. We firstly remark that from \eqref{eq4-05-30-2020}, 
$$
I_{t,x}^2(\lambda) = 2\lambda^{-3/2-\mu}\int_{1-r/t-\lambda}^{\lambda-t^{-1}}\int_0^{\theta_0}
\frac{\rho(\lambda-\rho)^{-1/2+\nu}d\rho\, d\theta}
{\sqrt{(1-\lambda)^2 - \big(\rho^2+(r/t)^2-2\rho(r/t)\cos\theta \big)}} .
$$
Then we make the following calculation:
$$
\aligned
&\int_0^{\theta_0}\Big((1-\lambda)^2 - \big(\rho^2+(r/t)^2-2\rho(r/t)\cos\theta\big)\Big)^{-1/2}d\theta
\\
=& 2(r/t)^{-1/2}\rho^{-1/2}\int_0^{\theta_0}\Big(\frac{(1-\lambda)^2 - (r/t)^2 - \rho^2}{2(r/t)\rho} + \cos\theta\Big)^{-1/2}\ d\theta
\\
=&2(r/t)^{-1/2}\rho^{-1/2}\int_0^{\theta_0}(\cos\theta-\cos\theta_0)^{-1/2}d\theta.
\endaligned
$$
Then apply the technical lemma \ref{lem1-31-05-2020} state and proved in the next subsection. To do so we need to determine whether $\theta_0\geq \pi/3$, that is,
$$
\aligned
\theta_0\geq \pi/3 &\Leftrightarrow \cos\theta_0 = -\frac{(1-\lambda)^2-(r/t)^2-\rho^2}{2(r/t)\rho} \leq 1/2 
\\
&\Leftrightarrow(1-\lambda)^2 - (r/t)^2 + (r/t)\rho - \rho^2\geq 0.
\endaligned
$$
This can be guaranteed by the following observations. First, $1-\lambda\geq 1 - \frac{t-r}{t} = r/t\geq 0$. Second, $\rho\leq \lambda-t^{-1}\leq 1-r/t \leq r/t\Rightarrow (r/t)\rho \geq \rho^2$. So we conclude that
\begin{equation}\label{eq1-01-06-2020}
\frac{t-r+1}{2t} \leq\lambda\leq 1-r/t,\ \rho\leq \lambda-t^{-1}\Rightarrow \theta_0\geq \pi/3.
\end{equation}
So by \eqref{eq4-01-06-2020}, 
$$
I_{t,x}^2(\lambda)\leq \bar{C}\lambda^{-3/2-\mu}\int_{1-r/t-\lambda}^{\lambda-t^{-1}}\rho^{1/2}(\lambda -\rho)^{-1/2+\nu}(\sin\theta_0)^{-1/2}d\rho.
$$
Remark that (because $\theta_0\geq \pi/3\Rightarrow \cos\theta_0\leq 1/2$)
$$
\aligned
(\sin\theta_0)^{-1/2} =& (1-\cos\theta_0)^{-1/4}(1+\cos\theta_0)^{-1/4} 
\\
\leq& \bar{C}\Big(\frac{(\rho+r/t)^2-(1-\lambda)^2}{2(r/t)\rho}\Big)^{-1/4}
\\
\leq& \bar{C}\rho^{1/4}(\rho-(1-\lambda-r/t))^{-1/4}(\rho+1-\lambda+r/t)^{-1/4}
\\
\leq& \bar{C}\rho^{1/4}(\rho - (1-\lambda-r/t))^{-1/4}.
\endaligned
$$
Here we have remarked that in the case of $I_{t,x}^2$, $\rho\geq 1-\lambda-r/t$, thus
$$
\rho + 1 -\lambda +r/t\geq 2(1+r/t-\lambda)\geq 4r/t\geq 1.
$$
So we conclude that
$$
\aligned
I_{t,x}^2(\lambda)\leq& \bar{C}\lambda^{-3/2-\mu}\int_{1-r/t-\lambda}^{\lambda-t^{-1}}
\rho^{3/4}(\lambda -\rho)^{-1/2+\nu}(\rho - (1-r/t-\lambda))^{-1/4}d\rho
\\
\leq& \bar{C}\lambda^{-3/4-\mu}(\lambda-t^{-1})^{3/4}\int_{1-r/t-\lambda}^{\lambda-t^{-1}}
(\lambda -\rho)^{-1/2+\nu}(\rho - (1-r/t-\lambda))^{-1/4}d\rho.
\endaligned
$$

Then apply Lemma \ref{lem1-01-06-2020} with $\alpha = 1/2-\nu, \beta = 1/4, a = 1-r/t-\lambda, b=\lambda-t^{-1}$, 

\begin{equation}\label{eq6-01-06-2020}
I_{t,x}^2(\lambda)\leq \bar{C}\lambda^{-\mu}\Big(\lambda - \frac{t-r+1}{2t}\Big)^{1/4+\nu}.
\end{equation}

Now we calculate, by \eqref{eq3-30-05-2020} and \eqref{eq6-01-06-2020},
\begin{equation}\label{eq4-31-05-2020}
\aligned
\int_{\frac{t-r+1}{2t}}^{1-r/t}I_{t,x}(\lambda)d\lambda
\leq& \bar{C}\int_{\frac{t-r+1}{2t}}^{1-r/t}\lambda^{-3/2-\mu}
\Big(\lambda - \frac{t-r}{2t}\Big)^{-1/2+\nu}(1-r/t-\lambda)^{3/2}d\lambda  
\\
& + \bar{C}\int_{\frac{t-r+1}{2t}}^{1-r/t} \lambda^{-\mu}\Big(\lambda - \frac{t-r+1}{2t}\Big)^{1/4 + \nu}d\lambda
\\
\leq& \bar{C}\big(\frac{t-r}{t}\big)^{-3/2-\mu}\Big(\frac{t-r}{t}\Big)^{3/2}
\int_{\frac{t-r+1}{2t}}^{1-r/t}\Big(\lambda - \frac{t-r}{2t}\Big)^{-1/2 + \nu}d\lambda
\\
& + \bar{C}\Big(\frac{t-r}{t}\Big)^{-\mu}\int_{\frac{t-r+1}{2t}}^{1-r/t}
\Big(\lambda - \frac{t-r+1}{2t}\Big)^{1/4+\nu}d\lambda 
\\
\leq& \bar{C}\Big(\frac{t-r}{t}\Big)^{1/2 -\mu+\nu}\leq \bar{C}(s/t)^{1+2(\nu-\mu)}. 
\endaligned
\end{equation}

\subsubsection{Case III: $B_O(\lambda)\not\subset B_A(\lambda), O\notin B_A(\lambda)$}
In this case, $1- r/t \leq\lambda\leq\frac{t+r+1}{2t}$.  Remark that 
$$
\lambda -  \frac{t-r}{t} \leq \rho \leq \lambda-t^{-1} 
$$
and for all $(\lambda, \rho)$ in this case, $\theta\in(-\theta_0,\theta_0)$ with $\theta_0$ defined in \eqref{eq2-31-05-2020}. Then
$$
I_{t,x}(\lambda) = \lambda^{-3/2-\mu} 
\int_{\lambda - \frac{t-r}{t}}^{\lambda-t^{-1}} 
\int_{-\theta_0}^{\theta_0} \frac{\rho(\lambda-\rho)^{-1/2+\nu} d\theta d\rho}{\sqrt{(1-\lambda)^2 - \rho^2 -(r/t)^2 + 2(r/t)\rho\cos\theta}}.
$$

We firstly show that when $\lambda \geq 1-r/t$, $\theta_0 <\pi/2$ (this will also be applied in the next case). This is because 
$$
\theta_0< \pi/2\ \Leftrightarrow \cos\theta_0 >0\ \Leftrightarrow\ \rho^2 + (r/t)^2 > (1-\lambda)^2.
$$
Remark that $\lambda - \frac{t-r}{t} \leq\rho<\lambda-t^{-1}$, we obtain:
$$
\rho^2 + (r/t)^2 \geq (\lambda - 1 + r/t)^2 + (r/t)^2 = (1-\lambda)^2 - 2(r/t)(1-\lambda) + 2(r/t)^2.
$$
Recall that $\lambda > 1-(r/t)\Rightarrow 1-\lambda < r/t$, 
$$
\rho^2 + (r/t)^2\geq (1-\lambda)^2 - 2(r/t)(1-(1-r/t)) +2(r/t)^2 >(1-\lambda)^2.
$$

Then similar to the bound of $I_{t,x}^2$ in the last case,  
$$
I_{t,x}(\lambda) \leq \bar{C}(r/t) \lambda^{-3/2-\mu} \int_{\lambda-\frac{t-r}{t}}^{\lambda-t^{-1}}\rho^{1/2}(\lambda-\rho)^{-1/2 + \nu} 
\int_0^{\theta_0}\big(\cos\theta - \cos\theta_0\big)^{-1/2}d\theta \ d\rho.
$$
So we apply \eqref{eq5-01-06-2020} and conclude that
\begin{equation}\label{eq3-31-05-2020}
I_{t,x}(\lambda)\leq \bar{C}\Big(\frac{t-r}{t}\Big)^{1/2+\nu}\lambda^{-1-\mu}.
\end{equation}
Integrate this bound on $\lambda\in[\frac{t-r}{t}, \frac{t+r+1}{2t}]$, one obtains:
\begin{equation}\label{eq7-31-05-2020}
\int_{\frac{t-r}{t}}^{\frac{t+r+1}{2t}}I_{t,x}(\lambda)d\lambda\leq \bar{C}\mu^{-1}\Big(\frac{t-r}{t}\Big)^{1/2+\nu}\Big(\frac{t-r}{t}\Big)^{-\mu}\leq \bar{C}\mu^{-1}(s/t)^{1+2(\nu-\mu)}.
\end{equation}

\subsubsection{Case IV: $B_O(\lambda) \supset B_A(\lambda)$} 
In this case $\frac{t+r+1}{2t}\leq \lambda \leq 1$. Then $\lambda - \frac{t-r}{t}\leq \rho\leq \frac{t+r}{t}-\lambda $, $\theta\in(-\theta_0,\theta_0)$. From the above case, we know that $\theta_0< \pi/2$.
$$
I_{t,x}(\lambda) = \lambda^{-3/2-\mu}\int_{\lambda - \frac{t-r}{t}}^{\frac{t+r}{t}-\lambda} 
\int_{-\theta_0}^{\theta_0}\frac{\rho(\lambda-\rho)^{-1/2+\nu} d\theta d\rho}{\sqrt{(1-\lambda)^2 - (r/t)^2 - \rho^2 + 2(r/t)\rho\cos\theta}}.
$$
Then as the above cases, by \eqref{eq5-01-06-2020} and remark that $9/10\leq r/t\leq 1$, 
$$
I_{t,x}(\lambda) \leq
\bar{C}(r/t)\lambda^{-3/2-\mu}\int_{\lambda-\frac{t-r}{t}}^{\frac{t+r}{t}-\lambda}\rho^{1/2}(\lambda-\rho)^{-1/2+\nu}d\rho
\leq \bar{C}\Big(\lambda  - \frac{t+r}{2t}\Big)^{-1/2+\nu}.
$$
Here we remark that $\lambda - \frac{t-r}{t}\geq \frac{t+r+1}{2t} - \frac{t-r}{t} = \frac{3r-t+1}{2t}\geq 17/20>1/2$.
So we conclude that 
\begin{equation}\label{eq5-31-05-2020}
\int_{\frac{t+r+1}{2t}}^1I_{t,x}(\lambda)d\lambda\leq \bar{C}\Big(\frac{t-r}{t}\Big)^{1/2+\nu}\leq \bar{C}(s/t)^{1+2(\nu-\mu)}.
\end{equation}

Now we recall \eqref{eq6-31-05-2020}, \eqref{eq4-31-05-2020}, \eqref{eq7-31-05-2020} combined with the above bound, the desired Lemma \ref{lem1-30-05-2020} is established.

\subsubsection{Technical lemmas}
\begin{lemma}\label{lem1-31-05-2020}
	Let $\theta_0\in(0,\pi)$. Then
	\begin{equation}\label{eq4-01-06-2020}
	\int_0^{\theta_0}(\cos\theta  - \cos\theta_0)^{-1/2} d\theta <
	\bar{C}(\sin\theta_0))^{-1/2},\quad \pi/3\leq \theta_0<\pi,
	\end{equation}
	\begin{equation}\label{eq5-01-06-2020}
	\int_0^{\theta_0}(\cos\theta  - \cos\theta_0)^{-1/2} d\theta < \bar{C},\quad 0<\theta_0\leq \pi/2.
	\end{equation}
	where $\bar{C}$ is a universal constant which does not depend on $\theta_0$.
\end{lemma}
\begin{proof}
	Remark that
	$$
	\aligned
	&\int_0^{\theta_0}(\cos\theta  - \cos\theta_0)^{-1/2} d\theta 
	\\
	=& 2^{-1/2}\int_0^{\theta_0}\Big(\sin\frac{\theta_0+\theta}{2}\Big)^{-1/2}
	\Big(\sin\frac{\theta_0-\theta}{2}\Big)^{-1/2}d\theta.
	\endaligned
	$$
	\\
-	When $\pi/3 \leq \theta_0<\pi$, $\pi/6 \leq \frac{\theta_0+\theta}{2}< \theta_0$. Then when $\frac{\theta_0+\theta}{2}\geq \pi/2$, $\sin\frac{\theta_0+\theta}{2}\geq \sin\theta_0$. When $\frac{\theta_0+\theta}{2}\leq \pi/2$, $\sin\frac{\theta_0+\theta}{2}\geq \sin(\pi/6)\geq \frac{1}{2} \sin\theta_0$. Then
	$$
	\Big(\sin\frac{\theta_0+\theta}{2}\Big)^{-1/2}\leq \bar{C}(\sin\theta_0)^{-1/2}.
	$$
	On the other hand, 
	$$
	0 < (\theta_0-\theta)/2 \leq \pi/2\ \Rightarrow\ \sin\frac{\theta_0 - \theta}{2}\geq \bar{C}^{-1}(\theta_0-\theta)
	$$
	where $\bar{C}>0$ is a universal constant. Then
	$$
	\int_0^{\theta_0}(\cos\theta  - \cos\theta_0)^{-1/2} d\theta \leq \bar{C}(\sin\theta_0)^{-1/2}\int_0^{\theta_0}\Big(\sin\frac{\theta_0-\theta}{2}\Big)^{-1/2}d\theta\leq \bar{C}(\sin\theta_0)^{-1/2}.
	$$
	
	\noindent - When $0< \theta_0< \pi/2$, $\theta_0/2\leq (\theta_0 + \theta)/2\leq \pi/2$. Thus
	$$
	\sin\frac{\theta_0+\theta}{2}\geq \sin\frac{\theta_0}{2},
	$$
	then
	$$
	\aligned
	\int_0^{\theta_0}(\cos\theta  - \cos\theta_0)^{-1/2} d\theta 
	\leq& 
	\bar{C}\Big(\sin\frac{\theta_0}{2}\Big)^{-1/2}\int_0^{\theta_0}\Big(\sin\frac{\theta_0-\theta}{2}\Big)^{-1/2}d\theta
	\\
	\leq& \bar{C}\Big(\sin\frac{\theta_0}{2}\Big)^{-1/2}\int_0^{\theta_0}(\theta_0-\theta)^{-1/2}d\theta
	\\
	\leq& \bar{C}\Big(\sin\frac{\theta_0}{2}\Big)^{-1/2}\theta_0^{1/2}\leq \bar{C}.
	\endaligned
	$$
\end{proof}
\begin{lemma}\label{lem1-01-06-2020}
	Let $a<b$, $1 > \alpha,\beta \geq 0$. Then
	$$
	\int_a^b(x-a)^{-\alpha}(x-b)^{-\beta}dx\leq 2(1-\alpha)^{-\alpha}(1-\beta)^{-\beta}(2-\alpha-\beta)^{\alpha+\beta-1} (b-a)^{1-\alpha-\beta}
	$$
\end{lemma}
\begin{proof}
	$$
	\aligned
	\int_a^b(x-a)^{-\alpha}(x-b)^{-\beta}dx =& \int_a^c(x-a)^{-\alpha}(b-x)^{-\beta}dx + \int_c^b(x-a)^{-\alpha}(b-x)^{-\beta}
	\\
	\leq& (b-c)^{-\beta}\int_a^c(x-a)^{-\alpha}dx + (c-a)^{-\alpha}\int_c^b(b-x)^{-\beta}
	\\
	\leq& \frac{(b-c)^{-\beta}(c-a)^{1-\alpha}}{1-\alpha} + \frac{(c-a)^{-\alpha}(b-c)^{1-\beta}}{1-\beta}.
	\endaligned
	$$
	Now taking $c = \frac{(1-\beta)a + (1-\alpha)b}{(1-\alpha) + (1-\beta)}$, the desired result is proved.
\end{proof}

\section{Decay bounds based on integration along hyperbolas}\label{sec2-08-06-2020}

\subsection{$L^{\infty}$ estimate on wave equation: differential identities}\label{subsec1-02-06-2020}
Suppose that $u$ is a function defined in $\Kcal = \{(t,x)\in\RR^{n+1}\ |\ r<t+1\}$, sufficiently regular. We make the following decomposition:
\begin{equation}\label{decompo-wave-1}
\Box u = (s/t)^2\del_t\del_t u + \frac{2x^a}{t}\delu_a\del_t u - \sum_a\delu_a\delu_a u + (n-(r/t)^2)t^{-1}\del_tu.
\end{equation}
This can be written in the following form:
\begin{equation}\label{decompo-wave-2}
\aligned
\Box u =& (t-r)^{-\beta}t^{-\alpha}\big((s/t)^2\del_t 
+ (2x^a/t)\delu_a\big)((t-r)^{\beta}t^{\alpha}\del_t u) + p_{n,\alpha,\beta}(t,r)\del_tu 
\\
&- \sum_a \delu_a\delu_a u
\endaligned
\end{equation}
where
$$
p_{n,\alpha,\beta}(t,r) = \big((n-\alpha) - (\alpha+1)(r/t)^2 - \beta(t-r)t^{-1}\big)t^{-1}.
$$
In order to keep $p_{n,\alpha,\beta}$ positive, one needs:
$$
\alpha \leq \frac{n-1}{2},\quad \beta\leq n-\alpha.
$$
In this case $n=2$, $\alpha = 1/2$. Then
$$
p_{\beta}(t,r) := p_{2,1/2,\beta}(t,r) = \left(\frac{3(t+r)}{2t}-\beta\right)\frac{t-r}{t^2}.
$$
In the case of strong coupling, we take $\beta = 1/2$. Then
$$
p(t,r) := p_{1.1/2,1/2}(t,r) = \big(1 + (3r/2t)\big)\frac{t-r}{t^2}\geq \frac{t-r}{t^2}\simeq (s/t)^2t^{-1}.
$$

On the other hand, let us concentrate on the operator $(s/t)^2\del_t + (2x^a/t)\delu_a$. It can be written as
$$
(s/t)^2\del_t + (2x^a/t)\delu_a = \frac{t^2+r^2}{t^2}\del_t + \frac{2x^a}{t}\del_a = \frac{t^2+r^2}{t^2}\left(\del_t + \frac{2tx^a}{t^2+r^2}\del_a\right).
$$
Then \eqref{decompo-wave-2} is written as
\begin{equation}\label{decompo-wave-2.5}
\mathcal{J}\big((t-r)^{1/2}t^{1/2} \del_tu\big) + P\ \big((t-r)^{1/2}t^{1/2} \del_tu\big) = S^w[u] + \Delta^w[u]
\end{equation}
with
$$
\aligned
&\mathcal{J} := \del_t + \frac{2tx^a}{t^2+r^2}\del_a,
\\
&P(t,r) := \frac{t^2}{t^2+r^2}p(t,r) = \frac{t-r}{t^2+r^2}(1+(3r/2t)) \geq \frac{1}{4}(s/t)^2t^{-1},
\\
&S^w[u] := t^{1/2}(t-r)^{1/2}\frac{t^2\Box u}{t^2+r^2},\quad 
\Delta^w[u] := t^{1/2}(t-r)^{1/2}\frac{t^2\sum_a\delu_a\delu_a u}{t^2+r^2}.
\endaligned
$$
The above identity \eqref{decompo-wave-2.5} will be regarded as ODE satisfied by $(t-r)^{1/2}t^{1/2}\del_t u$ .

To make it more clear, we analyze the integral curve of the vector field $\mathcal{J} = \del_t + \frac{2tx^a}{t^2+r^2}\del_a $. It can be explicitly calculated. Let $(t_0,x_0)\in\Kcal$, then the integral curve $\gamma(t;t_0,x_0)$ with $\gamma(t_0;t_0,x_0) = (t_0,x_0)$ is written as
\begin{equation}\label{eq1-18-08-2020}
\aligned
&\gamma(t;t_0,x_0) = \big(\gamma^\alpha(t;t_0,x_0)\big)_{\alpha=0,1,2},
\\
& \gamma^0(t;t_0,x_0) = t,\quad
\gamma^a(t;t_0,x_0) = (x_0^a/r_0)\left(\sqrt{t^2+\frac{1}{4}C_0^2} - \frac{1}{2}C_0\right)
\endaligned
\end{equation}
where 
$$
C_0 = \frac{t_0^2-r_0^2}{r_0}.
$$
This is a (time like) hyperbola with center at $(0,-\frac{x_0^a}{2r_0}C_0)$ and hyperbolic radius $\frac{1}{2}C_0$.

\subsection{Decay bound on $\del_t u$}
Now we establish the following $L^{\infty}-L^{\infty}$ bound. 
\begin{proposition}\label{prpo2 wave-sharp}
	Let $u$ be a sufficiently regular function defined in $\Hcal_{[s_0,s_1]}$, vanishes near $\del \Kcal = \{r=t-1\}$. 
	Then the following bound holds:
	\begin{equation}\label{eq1-29-05-2020}
	\frac{\sqrt{2}}{2}|s\del_t u(t,x)|\leq s_0\|\del_t u\|_{L^{\infty}(\Hcal_{s_0})} 
	+  \bigg|\int_{s_0}^t W_{t,x}[u](\tau) e^{-\int_{\tau}^t P_{t,x}(\eta)d\eta} d\tau\bigg|
	\end{equation}
	where 
	$$
	W[u]_{t,x}(\tau) := S^w[u]\Big|_{\gamma(\tau;t,x)} + \Delta^w[u]\Big|_{\gamma(\tau;t,x)}
	$$
	and
	$$
	P_{t,x}(\tau) := P\Big|_{\gamma(\tau;t,x)}. 
	$$
\end{proposition}

Before the proof, we need to make several observations on the integral curve $\gamma(\cdot; t,x)$:
\\
1. $\forall (t,x) \in \Hcal_s\subset  \Hcal_{[s_0,s_1]}$, $\gamma(\cdot;t,x)$ is time-like. 
\\
2.  $\gamma(\cdot;t,x)$ intersects $\Hcal_s$ at $\gamma(t;t,x) = (t,x)$. There exists a $t_0$ such that 
$$
\gamma(t_0;t,x)\in \Hcal_{s_0}\cup \Kcal, \quad \forall \tau \in [t_0,t], \gamma(\tau;t,x)\in \Hcal_{[s_0,s]}
$$
and $t_0\geq s_0$.

\begin{proof}[Proof of Proposiotion \ref{prpo2 wave-sharp}]
	
	This is by integrating \eqref{decompo-wave-2.5} along the hyperbola $\gamma(\tau;t,x)$. 
	
	For a fixed $(t,x)\in \Hcal_{[s_0,s_1]}$, let 
	$$
	U_{t,x} (\tau) := t^{1/2}(t-r)^{1/2}\del_t u\Big|_{\gamma(\tau;t,x)}.
	$$
	Then \eqref{decompo-wave-2.5} is written as 
	$$
	U_{t,x}'(\tau) + P_{t,x}(\tau) U_{t,x}(\tau) = W[u]_{t,x}(\tau).
	$$
	Integrate  this ODE on $\tau\in[t_0,t]$, one obtains:
	$$
	U_{t,x}(t) = U_{t,x}(t_0)e^{-\int_{t_0}^tP_{t,x}(\eta) d\eta} + \int_{t_0}^tW_{t,x}[u](\tau) e^{-\int_{\tau}^tP_{t,x}(\eta) d\eta}d\tau. 
	$$
	Then remark that $U_{t,x}(t) = t^{1/2}(t-r)^{1/2}\del_t u(t,x)$ and $s = \sqrt{(t+r)(t-r)}$  with $0\leq r\leq t$,
	$$
	\frac{\sqrt{2}}{2}s|\del_t u(t,x)|\leq |U_{t,x}(t_0)| + \bigg|\int_{t_0}^tW_{t,x}[u](\tau) e^{-\int_{t_0}^tP_{t,x}(\eta) d\eta}d\tau\bigg|
	$$
	where we applied the fact that $P(t,r)\geq 0$ in $\Hcal_{[s_0,s_1]}$. Now remark that 
	$$
	U_{t,x}(t_0) = t_0^{1/2}(t_0-r_0)^{1/2}\del_t u(t_0,x_0)
	$$
	with $r_0 = \sqrt{\sum_a|\gamma^a(t_0;t,x)|^2}$. Then there is two cases:
	\\
	- if $\gamma(t_0;t,x)\in\Hcal_{s_0}$, then 
	$$
	|U_{t,x}(t_0)|\leq s_0\sup_{\Hcal_{s_0}}\{|\del_t u|\},
	$$
	\\
	- if $\gamma(t_0;t,x)\in\Kcal = \{r = t-1\}$, then $\del_t u\big|_{\gamma(t_0;t,x)} = 0$. So $|U_{t,x}(t_0)| = 0$.
	
	Then the desired bound \eqref{eq1-29-05-2020} is established. 
\end{proof}
\section{Initialization of Bootstrap argument and the improvement of wave energy bounds}
\label{sec-bootstrap-1}
\subsection{Bootstrap bounds}

Let us suppose that for $s\in [2,s_1]$, the following {\bf bootstrap assumptions} hold with $C_1$ and $N\geq 14$:
\begin{equation}\label{eq1-26-04-2020}
\Ecal_2^N(s,u)^{1/2}\leq C_1\vep s^{1+\delta}, \quad \Ecal_{0,c}^N(s,v)^{1/2}\leq C_1\vep s^{\delta},
\end{equation}
\begin{equation}\label{eq2-26-04-2020}
\Ecal_2^{N-1}(s,u)^{1/2}\leq C_1\vep s,\quad \Ecal_{0,c}^{N-1}(s,v)^{1/2}\leq C_1\vep ,
\end{equation}
\begin{equation}\label{eq17-03-06-2020}
|\del_t\del^IL^J u|\leq C_1\vep s^{-1},\quad |I|+|J|\leq N-6.
\end{equation}
Here $C_1$ is a constant to be determined latter. When
\begin{equation}\label{eq3-14-09-2020}
\max\Big\{\Ecal_2^N(2,u)^{1/2}, \Ecal_2^{N-1}(2,u)^{1/2}, \sup_{\Hcal_2}\{|\del_t\del^IL^J u|\}\Big\}\leq C_0\vep < C_1\vep,
\end{equation}
by continuity, such interval $[2,s_1]$ exists.

Based on these bounds, we will prove that for $s\in[2,s_1]$ the following {\bfseries improved bounds} hold: 
\begin{equation}\label{eq1'-26-04-2020}
\Ecal_2^N(s,u)^{1/2}\leq \frac{1}{2}C_1\vep s^{1/2+\delta}, \quad \Ecal_{0,c}^N(s,v)^{1/2}\leq  \frac{1}{2}C_1\vep s^{\delta},
\end{equation}
\begin{equation}\label{eq2'-26-04-2020}
\Ecal_2^{N-1}(s,u)^{1/2}\leq  \frac{1}{2}C_1\vep s^{1/2},\quad \Ecal_{0,c}^{N-1}(s,v)^{1/2}\leq \frac{1}{2} C_1\vep ,
\end{equation}
\begin{equation}\label{eq3'-02-06-2020}
|\del_t\del^IL^J u|\leq  \frac{1}{2}C_1\vep s^{-1},\quad |I|+|J|\leq N-6.
\end{equation}
Then by classical bootstrap argument, we conclude that the local solution extends to time infinity and satisfies the above energy and decay bounds \eqref{eq1-26-04-2020}, \eqref{eq2-26-04-2020} and \eqref{eq17-03-06-2020}. The details of bootstrap argument can be found for example in \cite{Sogge-2008-book}.

\subsection{Basic $L^2$ bounds and Sobolev decay}

Based on \eqref{eq1-10-06-2020}, \eqref{eq5-10-06-2020}, \eqref{eq2-10-06-2020}, \eqref{eq6-10-06-2020} and the bootstrap assumptions \eqref{eq1-26-04-2020} and \eqref{eq2-26-04-2020}, the following bounds are direct:
\begin{equation}\label{eq4-15-06-2020}
\aligned
\|(s/t)^2 |\del u|_p\|_{L^2(\Hcal_s)}& + \||\dels u|_p\|_{L^2(\Hcal_s)} + s^{-1}\|(s/t)|u|_p\|_{L^2(\Hcal_s)}
\\
\leq& \left\{
\aligned
&CC_1\vep s^{\delta}, && p=N,
\\
&CC_1\vep ,&& p=N-1,
\endaligned
\right.
\endaligned
\end{equation}
\begin{equation}\label{eq5-15-06-2020}
\|(s/t)|\del v|_p\|_{L^2(\Hcal_s)} + \||\dels v|_p\|_{L^2(\Hcal_s)} + \||v|_p\|_{L^2(\Hcal_s)}\leq 
\left\{
\aligned
&CC_1\vep s^{\delta},\quad &&p=N,
\\
&CC_1\vep,&& p=N-1.
\endaligned
\right.
\end{equation}
Then by \eqref{eq3-10-06-2020}, \eqref{eq4-10-06-2020} combined with \eqref{eq1-26-04-2020} and \eqref{eq2-26-04-2020}
\begin{equation}\label{eq6-15-06-2020}
\aligned
\|(s/t)s|\del u|_{p-2}\|_{L^{\infty}(\Hcal_s)}& + \|(s/t)^{-1}s|\dels u|_{p-2}\|_{L^\infty(\Hcal_s)} + \||u|_{p-2}\|_{L^\infty(\Hcal_s)}
\\
\leq&
\left\{
\aligned
& CC_1\vep s^{\delta},\quad &&p = N,
\\
&CC_1\vep ,\quad && p=N-1,
\endaligned
\right.
\endaligned
\end{equation}
\begin{equation}\label{eq3-15-06-2020}
\aligned
\|s|\del v|_{p-2}\|_{L^{\infty}(\Hcal_s)}& + \|(s/t)^{-1}s|\dels v|_{p-2}\|_{L^\infty(\Hcal_s)} + \|(s/t)^{-1}s|v|_{p-2}\|_{L^\infty(\Hcal_s)}
\\
\leq&
\left\{
\aligned
& CC_1\vep s^{\delta},\quad &&p = N,
\\
&CC_1\vep ,\quad && p=N-1.
\endaligned
\right.
\endaligned
\end{equation}

\subsection{Bounds form \eqref{eq17-03-06-2020}}
From the assumption \eqref{eq17-03-06-2020} one can establish stronger decay on lower-order quantities, which are necessary in the following calculation. We firstly remark that for $|I|+|J|\leq N-6$,
$$
\del_a\del^IL^J u = \delu_a\del^IL^J u - (x^a/t)\del_t\del^IL^J u
$$
thus by \eqref{eq17-03-06-2020} and \eqref{eq6-15-06-2020}
$$
|\del_{\alpha}\del^IL^J u|\leq CC_1\vep t^{-1} + CC_1\vep s^{-1} \leq CC_1\vep s^{-1}.
$$
So by \eqref{eq11-10-06-2020} we obtain:
\begin{equation}\label{eq4-06-05-2020}
|\del u|_{N-6}\leq CC_1\vep s^{-1}.
\end{equation}

In fact by the following trick the decay bounds on $|\dels u|$ can be improved. Recall \eqref{eq1 lem5 notation}, for $|I|+|J|\leq N-7$, 
$$
|\del_r\del^IL^J\delu_a u|(t,x)\leq CC_1\vep (s/t)s^{-2} \leq CC_1\vep (t-r)^{-1/2}t^{-3/2}.
$$
Integrate this along redial direction,
$$
\del^IL^J\delu_au(t,x) = -\int_{|x|}^{t-1}\del_r(\del^IL^J\delu_a u)(t,\rho x/|x|)d\rho,
$$ 
then we obtain
\begin{equation}\label{eq5-06-05-2020}
|\dels u|_{N-7}\leq CC_1\vep (s/t)^2s^{-1}.
\end{equation}

\subsection{Fast decay of Klein-Gordon component near light-cone}
In order to recover the loss of conical decay, we need more precise decay on Klein-Gordon component. In this subsection we will establish the following bound:
\begin{equation}\label{eq1-31-07-2020}
|v|_p\leq
\left\{
\aligned
&CC_1\vep (s/t)^2s^{-1},\quad &&p=N-4,
\\
&CC_1\vep (s/t)^4s^{-1} + C(C_1\vep)^2(s/t)s^{-2},\quad &&p=N-6.
\endaligned
\right.
\end{equation}
This is based on Proposition \ref{prop1-fast-kg}, 
$$
c^2|v|_{N-4}\leq C(s/t)^2|\del v|_{N-3} + C|F_2|_{N-4}.
$$
Here remark that $|\del v|_{N-3}\leq CC_1\vep s^{-1}$ due to \eqref{eq3-15-06-2020}. For the bound on $F_2$, recall that $A_5$ is null and \eqref{eq13-10-06-2020} :
$$
\aligned
|A_5^{\alpha\beta}\del_{\alpha}u\del_{\beta}u|_{N-4} \leq& C(s/t)^2|\del u|_{N-4}|\del u|_{N-6} + |\dels u|_{N-4}|\del u|_{N-6} + |\dels u|_{N-7}|\del u|_{N-4}
\\
\leq& C(C_1\vep)^2(s/t)s^{-2} \leq C(C_1\vep)^2(s/t)^2s^{-1}.
\endaligned
$$
Here \eqref{eq4-06-05-2020} and \eqref{eq6-15-06-2020} are applied. This concludes the case of $p=N-4$.
Furthermore, apply the bound of order $N-4$,
\begin{equation}
|\del u|_{N-5}\leq C|v|_{N-4} \leq CC_1\vep(s/t)^2s^{-1} .
\end{equation}
Then \eqref{eq1-31-07-2020} is concluded.

\subsection{Improved energy bounds on wave component}
Equipped with \eqref{eq4-06-05-2020} and \eqref{eq5-06-05-2020}, the $L^2$ bound on $F_1$ becomes trivial. By \eqref{eq13-10-06-2020}
$$
\aligned
&\||A_1^{\alpha\beta}\del_{\alpha}u\del_{\beta}u|_p\|_{L^2(\Hcal_s)}
\\
\leq& C\|(s/t)^2|\del u|_{N-7}|\del u|_p\|_{L^2(\Hcal_s)} 
+ C\||\del u|_{N-7}|\dels u|_p\|_{L^2(\Hcal_s)}
\\ 
&+ C\|\dels u|_{N-7}|\del u|_p\|_{L^2(\Hcal_s)}
\\
\leq& CC_1\vep s^{-1}\|(s/t)^2|\del u|_p\|_{L^2(\Hcal_s)} 
+ CC_1\vep s^{-1}\||\dels u|_p\|_{L^2(\Hcal_s)} 
\\
&+ CC_1\vep s^{-1}\|(s/t)^2|\del u|_p\|_{L^2(\Hcal_s)}.
\endaligned
$$
Then by \eqref{eq4-15-06-2020}, 
\begin{equation}\label{eq4-18-08-2020}
\||A_1^{\alpha\beta}\del_{\alpha}u\del_{\beta}u|_p\|_{L^2(\Hcal_s)}\leq 
\left\{
\aligned
&C(C_1\vep)^2s^{-1+\delta},\quad && p=N,
\\
&C(C_1\vep)^2s^{-1},\quad && p=N-1.
\endaligned
\right.
\end{equation}

For $A_3$, the estimate is easier because $v$ enjoys better decay than $u$ (comparing \eqref{eq3-15-06-2020} with \eqref{eq4-06-05-2020} and \eqref{eq5-06-05-2020}). 

For $A_4$, remark that (thanks to \eqref{eq1-31-07-2020})
\begin{equation}\label{eq3-18-08-2020}
\aligned
\||vA_4^{\alpha}\del_{\alpha}u|_p\|_{L^2(\Hcal_s)}
\leq& C\||v|_{N-4}|\del u|_p\|_{L^2(\Hcal_s)} + C\||v|_p|\del u|_{N-6}\|_{L^2(\Hcal_s)}
\\
\leq& CC_1\vep s^{-1}\big(\|(s/t)^2|\del u|_p\|_{L^2(\Hcal_s)} + \||v|_p\|_{L^2(\Hcal_s)}\big)
\\
\leq& \left\{
\aligned
& C(C_1\vep)^2s^{-1 + \delta}, \quad&& p=N,
\\
& C(C_1\vep)s^{-1},\quad && p=N-1.
\endaligned
\right.
\endaligned
\end{equation}

The pure Klein-Gordon term is even more trivial. For example:
$$
\aligned
\||\del_{\alpha}v\del_{\beta}v|_p\|_{L^2(\Hcal_s)}\leq& \||\del v|_{N-4}|\del v|_p\|_{L^2(\Hcal_s)}
\leq 
CC_1\vep s^{-1}\|(s/t)|\del v|_p\|_{L^2(\Hcal_s)}
\\
\leq & 
\left\{
\aligned
& C(C_1\vep)^2 s^{-1+\delta},\quad &&p=N,
\\
& C(C_1\vep)^2 s^{-1},&& p=N-1.
\endaligned
\right.
\endaligned
$$

So we conclude that 
\begin{equation}\label{eq7-15-06-2020}
\||F_1|_p\|_{L^2(\Hcal_s)}\leq 
\left\{
\aligned
&C(C_1\vep)^2 s^{-1+\delta},\quad && p=N,
\\
&C(C_1\vep)^2 s^{-1},\quad && p=N-1.
\endaligned
\right.
\end{equation}

Now substitute this bound into the conformal energy estimate \eqref{eq8-15-06-2020}, one obtains, for $|I|+|J|\leq p$,
$$
\aligned
E_2(s,\del^IL^J u)^{1/2}\leq& E_2(2,\del^IL^J u)^{1/2} + C(C_1\vep)^2 \int_2^s \tau\||F_1|_p\|_{L^2(\Hcal_\tau)}
\endaligned
$$
which leads to
\begin{equation}\label{eq9-15-06-2020}
\Ecal_2^p(s, u)^{1/2}\leq C_0\vep +
\left\{
\aligned
& C(C_1\vep)^2s^{1+\delta},\quad && p=N,
\\
& C(C_1\vep)^2s ,\quad && p=N-1.
\endaligned
\right.
\end{equation}

\section{Improvement of Klein-Gordon energy bounds}
\label{sec-bootstrap-kg}
\subsection{Improved energy bounds for order $N$}
This is quite similar to the the bound of $A_1$. Remark that $F_2 = A_5^{\alpha\beta}\del_{\alpha}u\del_{\beta}u$ is also a null form. Thus 
\begin{equation}\label{eq10-15-06-2020}
\||A_5^{\alpha\beta}\del_{\alpha}u\del_{\beta}u|_N\|_{L^2(\Hcal_s)}\leq C(C_1\vep)^2 s^{-1+\delta}.
\end{equation} 
Substitute this bound into \eqref{ineq 3 prop 1 energy}, we obtain
\begin{equation}\label{eq11-15-06-2020}
\Ecal_{0,c}^N(s,u)^{1/2} \leq C_0\vep + C(C_1\vep)^2s^{\delta}. 
\end{equation}

\subsection{Nonlinear transform for order $N-1$}
Contrary to the high-order case, the bound of lower order on Klein-Gordon component is much more delicate. This is due to the logarithmic loss. To overpass this difficulty we rely on an algebraic trick applied in \cite{Ka}. 
$$
\Box \big(v - c^{-2}A_5^{\alpha\beta}\del_{\alpha}u\del_{\beta}u\big) 
+  c^2\big(v - c^{-2}A_5^{\alpha\beta}\del_{\alpha}u\del_{\beta}u\big) = -c^{-2}A_5^{\alpha\beta}\Box(\del_{\alpha}u\del_{\beta}u).
$$
Then by a direct calculation and the wave equation of \eqref{eq-main},
\begin{equation}\label{eq12-15-06-2020}
\Box w + c^2w = -2c^{-2}A_5^{\alpha\beta}m^{\mu\nu}\del_{\alpha}\del_{\mu}u\del_{\beta}\del_{\nu}u - 2c^{-2}A_5^{\alpha\beta}\del_{\alpha}u\del_{\beta}F_1
\end{equation}
where $w := v - c^{-2}A_5^{\alpha\beta}\del_{\alpha}u\del_{\beta}u$. The advantage of this transform is that, now in right-hand-side of \eqref{eq12-15-06-2020}, the second term is cubic and the first term, containing Hessian form of wave component, also enjoy integrable $L^2$ bounds. To see this we firstly establish the $L^2$ and $L^\infty$ bounds on Hessian form of wave component in the coming subsection.

\subsection{Bounds on Hessian form of wave component}
We will prove that
\begin{equation}\label{eq2-28-07-2020}
\|(s/t)^3s|\del\del u|_{p-1}\|_{L^2(\Hcal_s)} \leq
\left\{
\aligned
&CC_1\vep s^{\delta},\quad && p=N,
\\
&CC_1\vep ,\quad && p=N-1,
\endaligned
\right.
\end{equation}
\begin{equation}\label{eq1-28-07-2020}
(s/t)^2|\del\del u|_{p-1}\leq\left\{
\aligned
&CC_1\vep s^{-2+\delta},\quad &&p=N-3,
\\
&CC_1\vep (s/t)s^{-2},\quad && p=N-6.
\endaligned
\right.
\end{equation}
These bounds are based on \eqref{eq2 lem Hessian-flat-zero}. We will firstly establish the pointwise bound. To see this one only need to give sufficient decay bound on $|F_1|_{p,k}$. In fact we will prove that
\begin{equation}\label{eq1-27-07-2020}
|F_1|_p\leq C(C_1\vep)^2(s/t)s^{-2} + CC_1\vep t^{-1}|\del u|_p
\leq\left\{
\aligned
& C(C_1\vep)^2s^{-2},\quad &&p=N-3,
\\
& C(C_1\vep)^2(s/t)s^{-2},\quad &&p=N-6.
\endaligned
\right.
\end{equation}
To prove this we need to check each term in $F_1$. For the term $A_1$, $A_3$ we need their null structure. By \eqref{eq13-10-06-2020},
$$
\aligned
|A_1^{\alpha\beta}\del_{\alpha}u\del_{\beta}u|_p\leq& C(s/t)^2|\del u|_p|\del u|_{N-3} 
+ C|\dels u|_{N-3}|\del u|_p 
\\
\leq& CC_1\vep(s/t)s^{-1}|\del u|_p + CC_1\vep t^{-1}|\del u|_p\leq CC_1\vep t^{-1}|\del u|_p
\endaligned
$$
where \eqref{eq6-15-06-2020} are applied (case $p=N-1$). 
$$
\aligned
|A_3\del_{\alpha}u\del_{\beta}v|_p
\leq& C(s/t)^2|\del u|_p|\del v|_{N-3}
+ C|\dels u|_p|\del v|_{N-4} + C|\dels u|_{N-3}|\del v|_p
\\
&+ C|\del u|_p|\dels v|_{N-3}
\\
\leq& CC_1\vep (s/t)^2s^{-1} |\del u|_p
+ CC_1\vep t^{-1}|\dels u|_p
+ CC_1\vep t^{-1}|\del v|_p
\\
& + CC_1\vep t^{-1}|\del u|_p
\\
\leq& CC_1\vep t^{-1}|\del u|_p + C(C_1\vep)^2(s/t)^2s^{-2}.
\endaligned
$$
For $A_4$, remark that
$$
\aligned
|v\del_{\alpha}v|_p\leq&
|v|_{N-3}|\del u|_p + |v|_p|\del u|_{N-6}\leq CC_1\vep t^{-1}|\del u|_p + C(C_1\vep)^2 (s/t)s^{-2}
\endaligned
$$

The pure Klein-Gordon terms $B_2, B_3, K_2$ are much easier. We only show the bound of $B_2$:
$$
|B_2^{\alpha\beta}\del_{\alpha}v\del_{\beta}v|_{N-3}\leq C|\del v|_{N-3}|\del v|_{N-4}\leq C(C_1\vep)^2(s/t)^2s^{-2}.
$$
Then we conclude by \eqref{eq1-27-07-2020}.

Now substitute \eqref{eq1-27-07-2020} into \eqref{eq2 lem Hessian-flat-zero}, we arrive at
$$
(s/t)^2|\del\del u|_{p-1}\leq CC_1\vep (s/t)s^{-2} + Ct^{-1}|\del u|_p
$$
which leads to \eqref{eq1-28-07-2020}.

The proof of \eqref{eq2-28-07-2020} is quite similar. We need to establish the following bound:
\begin{equation}\label{eq3-28-07-2020}
\|s(s/t)|F_1|_{N-1}\|_{L^2(\Hcal_s)}\leq C(C_1\vep)^2.
\end{equation}
In fact this is guaranteed by \eqref{eq7-15-06-2020} (remark that $(s/t)\leq 1$ in $\Kcal$). Then by \eqref{eq2 lem Hessian-flat-zero},
$$
\|(s/t)^3s |\del\del u|_{p-1}\|_{L^2(\Hcal_s)}\leq C\|(s/t)s|F_1|_{p-1}\|_{L^2(\Hcal_s)} 
+ C\|(s/t)^2|\del u|_p\|_{L^2(\Hcal_s)}
$$
which leads to \eqref{eq2-28-07-2020}.

\subsection{Decay bound of $\del F_1$}
In order to bound the second term in right-hand-side of \eqref{eq12-15-06-2020} we need to establish the following bound:
\begin{equation}\label{eq3-15-08-2020}
|\del \tilde{F}_1|_{N-7}\leq C(C_1\vep)^2(s/t)^2s^{-2}
\end{equation}
where $\tilde{F}_1 = F_1 - A_4^{\alpha}v\del_{\alpha}u$, which is composed by the terms in $F_1$ except $A_4$, and
\begin{equation}\label{eq1-15-08-2020}
(s/t)^{-1}s^2|\del (A_4^{\alpha}v\del_{\alpha}u)|_{N-7} + (s/t)^{-2}s^3|\dels(A_4^{\alpha}v\del_{\alpha}u)|_{N-7}\leq C(C_1\vep)^2.
\end{equation}

This is also by checking each term. In fact by \eqref{eq13-10-06-2020},
$$
\aligned
\del_{\gamma}(A_1^{\alpha\beta}\del_{\alpha}u\del_{\beta}u) =& 2A_1^{\alpha\beta}\del_{\alpha}\del_{\gamma}u\del_{\beta}u 
= 2\Au_1^{\alpha\beta}\delu_{\alpha}\del_{\gamma}u\delu_{\beta}u 
\\
=& 2\Au_1^{00}\del_t\del_{\gamma}u\del_tu 
+ 2\sum_{(\alpha,\beta)\neq(0,0)}\Au_1^{\alpha\beta}\delu_{\alpha}\del_{\gamma}u\delu_{\beta}u
\endaligned
$$
Recall its null structure and by \eqref{eq4-06-05-2020},
$$
|A_1^{00}\del_t\del_{\gamma}u\del_tu|_{N-7}\leq C(C_1\vep)^2(s/t)^2|\del\del u|_{N-7}|\del u|_{N-7}\leq C(C_1\vep)^2(s/t)^2s^{-2}.
$$
The rest terms contain at least one good derivative. We remark that
$$
|\delu_a\del_{\gamma}u\del_{\beta}u|_{N-7}\leq C|\del\dels u|_{N-7}|\del u|_{N-7}
\leq Ct^{-1}|\del u|_{N-6}|\del u|_{N-7}\leq C(s/t)^2s^{-3}
$$
where \eqref{eq1 lem5 notation} is applied on $|\del\dels u|$ and \eqref{eq4-06-05-2020}, \eqref{eq5-06-05-2020} are applied on $\del u$ and $\dels u$ respectively.
$$
|\del_t\del_{\gamma}u\delu_b u|_{N-7}\leq C|\del\del u|_{N-7}|\dels u|_{N-7}\leq C(C_1\vep)^2(s/t)^2s^{-2}.
$$
where \eqref{eq4-06-05-2020} and \eqref{eq5-06-05-2020} are applied. So
\begin{equation}
|\del(A_1^{\alpha\beta}\del_{\alpha}u\del_{\beta}u)|_{N-7}\leq C(C_1\vep)^2(s/t)^2s^{-2}.
\end{equation}

The bound on $A_3$ is similar. We only need to remark that $|\del v|_p$ always enjoy better decay than $|\del u|_p$. Then 
$$
\aligned
&|A_3^{\alpha\beta}\del_{\alpha}u\del_{\beta}\del_{\gamma}v|_{N-7}
\\
\leq& C(s/t)^2|\del u|_{N-7}|\del\del v|_{N-7} + C|\dels u|_{N-7}|\del\del v|_{N-7} 
+ C|\del u|_{N-7}|\del\dels v|_{N-7}
\\
\leq& C(C_1\vep)^2(s/t)^2s^{-2} + C(C_1\vep)^2(s/t)^2s^{-3} + Ct^{-1}|\del u|_{N-7}|\del v|_{N-6}
\\
\leq& C(C_1\vep)^2(s/t)^2s^{-2}
\endaligned
$$
where \eqref{eq1 lem5 notation} , \eqref{eq4-06-05-2020} and \eqref{eq5-06-05-2020} are applied. 
$$
\aligned
&|A_3^{\alpha\beta}\del_{\alpha}\del_{\gamma}u\del_{\beta}v|_{N-7}
\\
\leq& C(s/t)^2|\del\del u|_{N-7}|\del v|_{N-7} + C|\del \dels u|_{N-7}|\del v|_{N-7} 
+ C|\del\del u|_{N-7}|\dels v|_{N-7}
\\
\leq& C(s/t)^2|\del\del u|_{N-6}|\del v|_{N-7} + Ct^{-1}|\del u|_{N-6}|\del v|_{N-7} 
+ Ct^{-1}|\del u|_{N-6}|v|_{N-6}
\\
\leq& C(C_1\vep)^2(s/t)^3s^{-2} + C(C_1\vep)^2(s/t)^2s^{-3} + C(s/t)^2s^{-3}
\\
\leq& C(C_1\vep)^2(s/t)^2s^{-2}.
\endaligned
$$
So we conclude that
\begin{equation}\label{eq2-15-08-2020}
|\del (A_3^{\alpha\beta}\del_{\alpha}u\del_{\beta}v)|_{N-7}\leq C(C_1\vep)^2(s/t)^2s^{-2}.
\end{equation}

The pure Klein-Gordon terms are easier. We only write the bound of $B_2$ in detail.
$$
|B_2^{\alpha\beta}\del_{\alpha}v\del_{\beta}\del_{\gamma}v|_{N-7}\leq C|\del v|_{N-7}|\del\del v|_{N-7}\leq C|v|_{N-5}^2\leq C(C_1\vep)^2(s/t)^2s^{-2}.
$$
Then we conclude by \eqref{eq3-15-08-2020}.

Finally we regard the bound of $A_4$. For this term we need to distinguish between good and bad derivatives. 
$$
|\del_{\alpha} (v\del_{\beta}u)|_{N-7} \leq  C|\del v|_{N-7}|\del u|_{N-7} + C|v|_{N-7}|\del\del u|_{N-7}\leq C(C_1\vep)^2(s/t)s^{-2}
$$
and
$$
|\delu_a (v\del_{\beta}u)|_{N-7}\leq C|\dels v|_{N-7}|\del u|_{N-7} + C|v|_{N-7}|\del\dels u|_{N-7}\leq C(C_1\vep)^2(s/t)^2s^{-3}.
$$
So we conclude by \eqref{eq1-15-08-2020}.

\subsection{Conclusion of this section}
We apply Proposition \ref{prop 1 energy} on \eqref{eq12-15-06-2020}. To do so we need to bound the $L^2$ norm of right-hand-side of \eqref{eq12-15-06-2020}. The first term is bounded as following:
\begin{equation}\label{eq1-19-08-2020}
\aligned
&A_5^{\alpha\beta}m^{\mu\nu}\,\del_{\alpha}\del_{\mu}u\del_{\beta}\del_{\nu}u
\\
=& \Au_5^{\alpha\beta}\minu^{\mu\nu}\delu_{\alpha}\delu_{\mu}u\, \delu_{\beta}\delu_{\nu}u
\\
&+ A_5^{\alpha\beta}m^{\mu\nu}\Psiu_{\nu}^{\nu'}\del_{\alpha}\big(\Psiu_{\mu}^{\mu'}\big)\delu_{\mu'}u\,\del_{\beta}\delu_{\nu'}u
+ A_5^{\alpha\beta}m^{\mu\nu}\Psiu_{\nu}^{\nu'}\del_{\alpha}\delu_{\mu'}u \del_{\beta}\big(\Psi_{\nu}^{\nu'}\big)\delu_{\nu'}u
\\
&+ A_5^{\alpha\beta}m^{\mu\nu}\del_{\alpha}\big(\Psiu_{\mu}^{\mu'}\big)\del_{\beta}\big(\Psiu_{\nu}^{\nu'}\big)\delu_{\mu'}u\delu_{\nu'}u.
\endaligned
\end{equation}
The last term contain decreasing factor $\del_{\alpha}\big(\Psiu_{\mu}^{\mu'}\big)\simeq t^{-1}$, thus can be bounded by $C(C_1\vep)^2s^{-2+\delta}$:
$$
\aligned
&\||A_5^{\alpha\beta}m^{\mu\nu}\del_{\alpha}\big(\Psiu_{\mu}^{\mu'}\big)\del_{\beta}\big(\Psiu_{\nu}^{\nu'}\big)\delu_{\mu'}u\delu_{\nu'}u|_{N-1}\|_{L^2(\Hcal_s)}
\\
\leq& C\|t^{-2}|\del_{\alpha} u\del_{\beta} u|_{N-1}\|_{L^2(\Hcal_s)}
\\
\leq& Cs^{-2}\|(s/t)^2|\del u|_{N-1}|\del u|_{N-6}\|_{L^2(\Hcal_s)}
\\
\leq& C(C_1\vep)^2s^{-3}.
\endaligned
$$

The second and the third term in right-hand-side of \eqref{eq1-19-08-2020} are critical. We firstly make the following calculation:
\begin{equation}\label{eq2-19-08-2020}
\aligned
&A_5^{\alpha\beta}m^{\mu\nu}\Psiu_{\nu}^{\nu'}\del_{\alpha}\big(\Psiu_{\mu}^{\mu'}\big)\delu_{\mu'}u\,\del_{\beta}\delu_{\nu'}u
\\
=&
\Au_5^{\alpha\beta}m^{\mu\nu}\Psiu_{\nu}^{\nu'}\delu_{\alpha}\big(\Psiu_{\mu}^{\mu'}\big)\delu_{\mu'}u\delu_{\beta}\delu_{\nu'}u
\\
=&\Au_5^{00}m^{\mu\nu}\Psiu_{\nu}^{\nu'}\del_t\big(\Psiu_{\mu}^{\mu'}\big)\delu_{\mu'}u\del_t\delu_{\nu'}u
+ \sum_{b}\Au_5^{\alpha b}m^{\mu\nu}\Psiu_{\nu}^{\nu'}\delu_{\alpha}\big(\Psiu_{\mu}^{\mu'}\big)\delu_{\mu'}u\,\delu_b\delu_{\nu'}u
\\
& + \sum_{a}\Au_5^{a0}m^{\mu\nu}\Psiu_{\nu}^{0}\delu_a\big(\Psiu_{\mu}^{c}\big)\delu_cu\del_t\del_tu
+ \sum_{a}\Au_5^{a0}m^{\mu\nu}\Psiu_{\nu}^{c}\delu_a\big(\Psiu_{\mu}^{\mu'}\big)\delu_{\mu'}u\del_t\delu_c u
\\
&+   \sum_{a}\Au_5^{a0}m^{\mu\nu}\Psiu_{\nu}^0\delu_a\big(\Psiu_{\mu}^{0}\big)\del_t u\del_t\del_t u.
\endaligned
\end{equation}
Thanks to the null condition on $A_5$, the first term is bounded as following:
$$
\aligned
\|\Au_5^{00}m^{\mu\nu}\Psiu_{\nu}^{\nu'}\del_t\big(\Psiu_{\mu}^{\mu'}\big)\delu_{\mu'}u\del_t\delu_{\nu'}u\|_{L^2(\Hcal_s)}
\leq& C\|(s/t)^2t^{-1}|\del u \del\del u|_{N-1}\|_{L^2(\Hcal_s)}
\\
\leq& C s^{-1}\|(s/t)^3|\del u|_N|\del u|_{N-6}\|_{L^2(\Hcal_s)}
\\
\leq& C(C_1\vep)^2s^{-2+\delta}.
\endaligned
$$
The second term in right-hand-side of \eqref{eq2-19-08-2020} contains a good derivative, thus can also be bounded directly:
$$
\aligned
&\||\Au_5^{\alpha b}m^{\mu\nu}\Psiu_{\nu}^{\nu'}\delu_{\alpha}\big(\Psiu_{\mu}^{\mu'}\big)\delu_{\mu'}u\,\delu_b\delu_{\nu'}u|_{N-1}\|_{L^2(\Hcal_s)}
\\
\leq& C\|t^{-1}|\del u \dels\del u|_{N-1}\|_{L^2(\Hcal_s)}
\\
\leq& Cs^{-1}\|(s/t)|\del u|_{N-1}|\del\dels u|_{N-7}\|_{L^2(\Hcal_s)} + Cs^{-1}\|(s/t)|\del u|_{N-6}|\del\dels u|_{N-1}\|_{L^2(\Hcal_s)}
\\
\leq& CC_1\vep s^{-1}\|(s/t)^2s^{-2}|\del u|_{N-1}\|_{L^2(\Hcal_s)} + CC_1\vep s^{-2} \|(s/t)t^{-1}|\del u|_N\|_{L^2(\Hcal_s)}
\\
\leq& C(C_1\vep)^2s^{-2+\delta}.
\endaligned
$$
The third and forth term are bounded in the same manner, we omit the detail. The last term is the most critical one. We need to dig more of its structure.  Remark that $\Psiu_0^0 = 1$ and $\Psiu_0^{c} = -x^c/t$, and $m^{ab} = - \delta^{ab}$. Then
$$
\aligned
&m^{\mu\nu}\Psiu_{\nu}^0\delu_a\big(\Psiu_{\mu}^{0}\big)\del_t u\del_t\del_t u 
\\
=& -\sum_{d=1}^2\Psiu_d^0\delu_a(\Psiu_d^0)\del_tu\del_t\del_tu
= \sum_{d=1}^2 (x^d/t)\big(x^ax^d/t^3 - \delta_a^d/t\big)\del_tu\del_t\del_tu
\\
=& (x^ar^2/t^4 - x^a/t^2)\del_tu\del_t\del_t u = (x^a/t^2)(r^2/t^2-1)\del_tu\del_t\del_tu
\\
=& - (x^a/t^2)(s/t)^2\del_tu\del_t\del_tu.
\endaligned
$$
This additional conical decay $(s/t)^2$ is crucial. Then 
$$
\aligned
&\||\Au_5^{a0}m^{\mu\nu}\Psiu_{\nu}^0\delu_a\big(\Psiu_{\mu}^{0}\big)\del_t u\del_t\del_t u|_{N-1}\|_{L^2(\Hcal_s)}
\\
\leq& C\|(s/t)^2t^{-1}|\del u \del\del u|_{N-1}\|_{L^2(\Hcal_s)}
\\
\leq& Cs^{-1}\|(s/t)^3|\del u|_{N-6}|\del\del u|_{N-1}\|_{L^2(\Hcal_s)} + Cs^{-1}\|(s/t)^3|\del u|_{N-1}|\del\del u|_{N-7}\|_{L^2(\Hcal_s)}
\\
\leq& CC_1\vep s^{-2}\|(s/t)^2|\del u|_N\|_{L^2(\Hcal_s)} \leq C(C_1\vep)^2s^{-2+\delta}.
\endaligned
$$ 
So we conclude that
\begin{equation}\label{eq3-19-08-2020}
\||A_5^{\alpha\beta}m^{\mu\nu}\Psiu_{\nu}^{\nu'}\del_{\alpha}\big(\Psiu_{\mu}^{\mu'}\big)\delu_{\mu'}u\,\del_{\beta}\delu_{\nu'}u|_{N-1}\|_{L^2(\Hcal_s)}\leq C(C_1\vep)^2s^{-2+\delta}.
\end{equation}

Now we regard the first term in right-hand-side of \eqref{eq1-19-08-2020}.
\begin{equation}\label{eq1-29-07-2020}
\aligned
&\Au_5^{\alpha\beta}\minu^{\mu\nu}\delu_{\alpha}\delu_{\mu}u\, \delu_{\beta}\delu_{\nu}u
\\
=& \Au_5^{00}\minu^{00}\del_t\del_tu\,\del_t\del_tu 
+ 2\Au_5^{00}\minu^{a0}\del_t\delu_au\,\del_t\del_t u + 2\Au_5^{a0}\minu^{00}\delu_a\del_tu\,\del_t\del_tu
\\
&+ \Au_5^{a0}\minu^{b0}\delu_a\delu_b u\, \del_t\del_t u + \Au_5^{a0}\minu^{0b}\delu_a\del_tu\,\del_t\delu_bu
\\
&+ \Au_5^{0a}\minu^{b0}\del_t\delu_bu\,\delu_a\del_t u + \Au_5^{0a}\minu^{0b}\del_t\del_t u\,\delu_a\delu_bu
\\
&+ \Au_5^{ab}\minu^{00}\delu_a\del_tu\,\delu_b\del_tu + \Au_5^{00}\minu^{ab}\del_t\delu_au\,\del_t\delu_bu
\\
&+ 2\Au_5^{ab}\minu^{c0}\delu_a\delu_cu\,\delu_b\del_tu + 2\Au_5^{a0}\minu^{bc}\delu_a\delu_bu\,\del_t\delu_c u
\\
&+ \Au_5^{ab}\minu^{cd}\delu_a\delu_cu\,\delu_b\delu_d u.
\endaligned
\end{equation}
Remark that both $A_5$ and $m$ are null quadratic forms. Thus $|\Au_5^{00}|_{p,k} + |\minu^{00}|_{p,k}\leq C(s/t)^2$. Then recall the bounds \eqref{eq2-28-07-2020} and \eqref{eq1-28-07-2020}:
$$
\aligned
\||\Au_5^{00}\minu^{00}\del_t\del_tu\,\del_t\del_tu|_{N-1}\|_{L^2(\Hcal_s)}
\leq& C\|(s/t)^4|\del\del u|_{N-7}|\del\del u|_{N-1}\|_{L^2(\Hcal_s)}
\\
\leq& CC_1\vep\|(s/t)^3s^{-2}|\del\del u|_{N-1}\|_{L^2(\Hcal_s)}
\\
\leq& CC_1\vep s^{-3}\|(s/t)^3s|\del\del u|_{N-1}\|_{L^2(\Hcal_s)}
\\
\leq& C(C_1\vep)^2s^{-3+\delta}.
\endaligned
$$
The rest terms are bounded in a similar way. We need to apply the following bounds due to \eqref{eq1 lem5 notation}, \eqref{eq4-15-06-2020} and \eqref{eq6-15-06-2020}:
\begin{equation}\label{eq2-29-07-2020}
\|(s/t)s|\del\dels u|_{N-1}\|_{L^2(\Hcal_s)} + \|s^2|\del\dels u|_{N-3}\|_{L^{\infty}(\Hcal_s)}\leq CC_1\vep s^{\delta}
\end{equation}
and, based on \eqref{eq4-06-05-2020},
\begin{equation}\label{eq3-29-07-2020}
|\del\dels u|_{N-7}\leq CC_1\vep (s/t)s^{-2}.
\end{equation}
Similarly, by \eqref{eq4 notation} combined with \eqref{eq4-15-06-2020} and \eqref{eq5-15-06-2020},
\begin{equation}\label{eq4-29-07-2020}
\|t|\dels\dels u|_{N-1}\|_{L^2(\Hcal_s)} + \|t^2|\dels\dels u|_{N-3}\|_{L^{\infty}(\Hcal_s)}\leq CC_1\vep s^{\delta}.
\end{equation}
Also, by \eqref{eq5-06-05-2020}
\begin{equation}\label{eq5-29-07-2020}
|\dels\dels u|_{N-8}\leq CC_1\vep (s/t)^3s^{-2}.
\end{equation}

Substitute these bounds into the corresponding expressions, we can prove that
$$
\aligned
&\||\del\dels u\, \del\dels u|_{N-1}\|_{L^2(\Hcal_s)}\leq C\||\del\dels u|_{N-3}|\del\dels u|_{N-1}\|_{L^2(\Hcal_s)}
\\
\leq& CC_1\vep s^{-2}\|(s/t)^{-1}s^{-1}(s/t)s|\del\dels u|_{N-1}\|_{L^2(\Hcal_s)} 
\leq C(C_1\vep)^2s^{-2 + \delta} 
\endaligned
$$
where we have applied the fact that $(s/t)\leq Cs^{-1}$ in $\Kcal$. In the same manner,
$$
\aligned
&\|(s/t)^2|\del\del u\,\del\dels u|_{N-1}\|_{L^2(\Hcal_s)}
\\
\leq& C\|(s/t)^2|\del\del u|_{N-7}|\del\dels u|_{N-1}\|_{L^2(\Hcal_s)} 
+ C\|(s/t)^2|\del\del u|_{N-1}|\del\dels u|_{N-7}\|_{L^2(\Hcal_s)}
\\
\leq& CC_1\vep s^{-3}\|(s/t)s|\del\dels u|_{N-1}\|_{L^2(\Hcal_s)}
+ CC_1\vep s^{-3}\|(s/t)^3s|\del\del u|_{N-1}\|_{L^2(\Hcal_s)}
\\
\leq& CC_1\vep s^{-3+\delta}.
\endaligned
$$
$$
\aligned
&\||\dels\dels u\, \del\del u|_{N-1}\|_{L^2(\Hcal_s)}
\\
\leq& C\||\dels\dels u|_{N-1}|\del\del u|_{N-3}\|_{L^2(\Hcal_s)} + C\||\dels\dels u|_{N-8}|\del\del u|_{N-1}\|_{L^2(\Hcal_s)}
\\
\leq& CC_1\vep\|s^{-1+\delta}t^{-1}\, t|\dels\dels u|_{N-1}\|_{L^2(\Hcal_s)} 
+ CC_1\vep \|(s/t)^3s^{-2}|\del\del u|_{N-1}\|_{L^2(\Hcal_s)}
\\
\leq&C(C_1\vep)^2s^{-2+2\delta}.
\endaligned
$$
With these bounds, we conclude that
\begin{equation}\label{eq6-29-07-2020}
\||\Au_5^{\alpha\beta}\minu^{\mu\nu}\delu_{\alpha}\delu_{\mu}u\, \delu_{\beta}\delu_{\nu}u|_{N-1}\|_{L^2(\Hcal_s)}
\leq C(C_1\vep)^2s^{-2+2\delta}
\end{equation}
which gives integrable bound for the first term in right-hand-side of \eqref{eq12-15-06-2020}.

Now we regard the second term in right-hand-side of \eqref{eq12-15-06-2020}. Recall the following decomposition.
\begin{equation}\label{eq4-15-08-2020}
A_5^{\alpha\beta}\del_{\alpha}u\del_{\beta}F_1 = A_5^{\alpha\beta}\del_{\alpha}u\del_{\beta}\big(A_4^{\gamma}v\del_{\gamma}u\big) + A_5^{\alpha\beta}\del_{\alpha}u\del_{\beta}\tilde{F}_1. 
\end{equation}
The second term in right-hand-side is easier:
$$
\aligned
&\||A_5^{\alpha\beta}\del_{\alpha}u\del_{\beta}\tilde{F}_1|_{N-1}\|_{L^2(\Hcal_s)}
\\
\leq& C\||\del u|_{N-6}|\tilde{F}_1|_{N-1}\|_{L^2(\Hcal_s)} + C\||\del u|_{N-1}|\tilde{F}_1|_{N-7}\|_{L^2(\Hcal_s)}
\\
\leq& CC_1\vep s^{-1}\||\tilde{F}_1|_{N-1}\|_{L^2(\Hcal_s)} + CC_1\vep s^{-2}\|(s/t)^2|\del u|_{N-1}\|_{L^2(\Hcal_s)}
\\
\leq& C(C_1\vep)^2 s^{-1}\||F_1|_{N-1}\|_{L^2(\Hcal_s)} + CC_1\vep s^{-2}\|(s/t)^2|\del u|_{N-1}\|_{L^2(\Hcal_s)}
\\
\leq& C(C_1\vep)^2 s^{-2+\delta}
\endaligned
$$
where we we have applied \eqref{eq7-15-06-2020} and \eqref{eq3-15-08-2020}.

For the first term in right-hand-side of \eqref{eq4-15-08-2020} we need to evoke the null structure of $A_5$. More precisely,
$$
\aligned
&\||A_5^{\alpha\beta}\del_{\alpha} u \del_{\beta}(A_4^{\gamma}v\del_{\gamma}u)|_{N-1}\|_{L^2(\Hcal_s)}
\\
\leq& C\|(s/t)^2|\del u|_{N-6}|\del (v\del_{\gamma} u)|_{N-1}\|_{L^2(\Hcal_s)} 
+ C\|(s/t)^2|\del u|_{N-1}|\del(v\del_{\gamma}u)|_{N-7}\|_{L^2(\Hcal_s)}
\\
& + C\||\dels u|_{N-7}|\del (v\del_{\gamma}u)|_{N-1}\|_{L^2(\Hcal_s)}
+C\||\dels u|_{N-1}|\del(v\del_{\gamma}u)|_{N-7}\|_{L^2(\Hcal_s)}
\\
& + C\||\del u|_{N-6}|\dels (v\del_{\gamma}u)|_{N-1}\|_{L^2(\Hcal_s)}
+C\||\del u|_{N-1}|\dels(v\del_{\gamma}u)|_{N-7}\|_{L^2(\Hcal_s)}
\\
\leq& CC_1\vep s^{-1}\| (s/t)^2|F_1|_{N-1} \|_{L^2(\Hcal_s)} 
+ CC_1\vep s^{-2}\| (s/t)^2|\del u|_{N-1}\|_{L^2(\Hcal_s)}\|_{L^2(\Hcal_s)}
\\
& + CC_1\vep s^{-1}\|(s/t)^2|F_1|_{N-1}\|_{L^2(\Hcal_s)}
+ CC_1\vep s^{-2}\|(s/t)|\dels u|_{N-1}\|_{L62(\Hcal_s)}
\\
&+ CC_1\vep s^{-1}\|t^{-1}|v\del_{\gamma} u|_N\|_{L^2(\Hcal_s)}
+ CC_1\vep s^{-3}\|(s/t)^2|\del u|_{N-1}\|_{L^2(\Hcal_s)}
\\
\leq& C(C_1\vep)^2s^{-2+\delta}.
\endaligned
$$
Here for the second inequality we have applied \eqref{eq4-06-05-2020}, \eqref{eq1-15-08-2020}, \eqref{eq5-06-05-2020}, \eqref{eq1-15-08-2020}, \eqref{eq4-06-05-2020} and \eqref{eq1-15-08-2020} respectively on each term. And then for the last inequality \eqref{eq7-15-06-2020}, \eqref{eq4-15-06-2020} are applied. Remark that this bound is integrable.

Then apply Proposition \ref{prop 1 energy} and obtain:
\begin{equation}\label{eq5-15-08-2020}
\Ecal_{0,c}^{N-1}(s,w)^{1/2}\leq C_0\vep + C(C_1\vep)^2.
\end{equation}
To recover the bound on $v$, we only need to recall \eqref{eq10-15-06-2020}. Then the following bound is established:
\begin{equation}\label{eq8-29-07-2020}
\Ecal_{0,c}^{N-1}(s,v)^{1/2}\leq C_0\vep + C(C_1\vep)^2.
\end{equation}

\section{Proof of \eqref{eq3'-02-06-2020}}
\label{sec-bootstrap-decay}
\subsection{Algebraic preparation}
\eqref{eq3'-02-06-2020} is the most critical one throughout this article. It relies on Proposition \ref{prpo2 wave-sharp} and Proposition \ref{prop1-01-06-2020}. However, we cannot apply them directly on $u$, because the term $A_1^{\alpha\beta}\del_{\alpha}u\del_{\beta}u$ will never have sufficient decay. Remark that $u$ is a scalar function, so we recall the following well-known transformation which will eliminate this term. Let $\phi = u - \frac{A_1^{00}}{2}u^2$. Then
\begin{equation}\label{eq12-30-07-2020}
\Box \phi = A_3^{\alpha\beta}\del_{\alpha}u\del_{\beta}v + A_4^{\alpha}v\del_{\alpha}u + B_2^{\alpha\beta}\del_{\alpha}v\del_{\beta}v + B_3^{\alpha}v\del_{\alpha}v + K_2v^2 - A_1^{00}uF_1.
\end{equation}
This is the only place that we demand $u$ is a scalar. To see this let us recall the following result:
\begin{lemma}\label{prop1-30-07-2020}
	Let $A^{\alpha\beta}$ be a symmetric null quadratic form, i.e,
	$$
	A^{\alpha\beta}\xi_{\alpha}\xi_{\beta} = 0,\quad \forall \xi\text{ satisfying } \xi_0^2 - \xi_1^2 - \xi^2_2 = 0 
	$$ 
	and
	$$
	A^{\alpha\beta} = A^{\beta\alpha}. 
	$$
	Then 
	$$
	A^{\alpha\beta} = A^{00}m^{\alpha\beta}
	$$
	where $m^{\alpha\beta}$ is the Minkowski metric. 
\end{lemma}
The proof of this Lemma is postponed into Appendix. Then we make the following calculation:
$$
\Box (u^2) = m^{\alpha\beta}\del_{\alpha}\del_{\beta}(u^2) 
= 2m^{\alpha\beta}\del_{\alpha}u\del_{\beta}u + 2u\Box u
= 2m^{\alpha\beta}\del_{\alpha}u\del_{\beta}u + 2uF_1.
$$
Equipped this identity, we calculate $\Box \phi$ and substitute \eqref{eq-main} into the expression. Then \eqref{eq12-30-07-2020} is established. Then we differentiate \eqref{eq12-30-07-2020} with respect to $\del^IL^J$ with $|I|+|J|\leq N-6$, 
\begin{equation}\label{eq1-21-08-2020}
\aligned
&\Box \del^IL^J \phi 
\\
=& \del^IL^J( A_3^{\alpha\beta}\del_{\alpha}u\del_{\beta}v + A_4^{\alpha}v\del_{\alpha}u + B_2^{\alpha\beta}\del_{\alpha}v\del_{\beta}v + B_3^{\alpha}v\del_{\alpha}v + K_2v^2 - A_1^{00}uF_1).
\endaligned
\end{equation}

We will apply Proposition \ref{prpo2 wave-sharp} on \eqref{eq1-21-08-2020} and then obtain the following bound:
\begin{equation}\label{eq13-30-07-2020}
|\del_t \del^IL^J \phi|_{N-6}\leq C_0\vep s^{-1}+ C(C_1\vep)^2s^{-1}.
\end{equation}
To do so we need to establish the following bounds for $|I|+|J|\leq N-6$ (following the notation of Proposition  \ref{prpo2 wave-sharp}):
\begin{equation}\label{eq1-30-07-2020}
|\Delta^w[\del^IL^J \phi]|\leq  C_2\vep(s/t)^3s^{-1} + (C_1\vep)^2 (s/t)^3s^{-1}
\end{equation}
with $C_2$ a constant determined by the system and $N$, and
\begin{equation}\label{eq2-30-07-2020}
|S^w[\del^IL^J \phi]|\leq C(C_1\vep)^2(s/t)^2t^{-1} + C(C_1\vep)^2(s/t)s^{-2}.
\end{equation}
The following subsections are devoted to these bounds.

\subsection{Proof of \eqref{eq1-30-07-2020}}
In the region $\Hcal_{[2,s_1]}^{\text{int}} := \{r\leq 9t/10\}\cap\Hcal_{[2,s_1]}$, this bound is directly by \eqref{eq4-10-06-2020} and \eqref{eq4 notation} combined with \eqref{eq9-15-06-2020}. More precisely,
\begin{equation}\label{eq5-14-09-2020}
|u|_{N-3}\leq Cs^{-1}\Fcal_2^{N-1}(2;s,u)\leq CC_0\vep s^{-1}\ln(s) + C(C_1\vep)^2\leq CC_1\vep.
\end{equation}
Then
$$
t^2|\delu_a\delu_a u|_{N-5}\leq CC_0\vep s^{-1}\ln (s) + C(C_1\vep)^2.
$$

Remark that when $r\leq 9t/10$, $(s/t)\geq \sqrt{19}/10>0$. Thus
\begin{equation}\label{eq4-08-09-2020}
|\delu_a\delu_a u|_{N-5}\leq C\big(C_0\vep + (C_1\vep)^2\big) (s/t)^3s^{-2}.
\end{equation}
Also recall that
$$
|\delu_a\delu_a(u^2)|_{N-5}\leq C(C_1\vep)^2 t^{-2}\leq C(C_1\vep)^2(s/t)^3s^{-2}.
$$
Then in the region $\Hcal_{[2,s_1]}^{\text{int}}$, \eqref{eq1-30-07-2020} is verified with $C_2\geq CC_0$.

When near the light-cone, we need to apply Proposition \ref{prop1-01-06-2020}. To do so we need to establish the following bound:
\begin{equation}\label{eq3-30-07-2020}
|F_1|_{N-4}\leq C(C_1\vep)^2 (s/t)s^{-2}\simeq C(C_1\vep)^2t^{-3/2}(t-r)^{-1/2}.
\end{equation}
This is again by checking each term in $F_1$. By \eqref{eq13-10-06-2020}, and especially \eqref{eq5-06-05-2020}
\begin{equation}\label{eq7-20-08-2020}
\aligned
&|A_1^{\alpha\beta}\del_{\alpha}u\del_{\beta}u|_{N-4}
\\
\leq&
C(s/t)^2|\del u|_{N-4}|\del u|_{N-6} + C|\del u|_{N-4}|\dels u|_{N-7} + C|\del u|_{N-6}|\dels u|_{N-4}
\\
\leq& 
C(C_1\vep)^2(s/t)s^{-2}
\leq C(C_1\vep)^2 t^{-3/2}(t-r)^{-1/2}.
\endaligned
\end{equation}
The rest terms enjoy better decay. We underline their precise decays and then write them in the form of \eqref{eq3-30-07-2020}. These precise decays will be applied latter.
\begin{equation}\label{eq4-20-08-2020}
\aligned
&|A_3^{\alpha\beta}\del_{\alpha}u\del_{\beta}v|_{N-4}
\\
\leq& C(s/t)^2|\del u|_{N-4}|\del v|_{N-4} 
+ C|\del u|_{N-4}|\dels v|_{N-4} + C|\dels u|_{N-4}|\del v|_{N-4}
\\
\leq& C(C_1\vep)^2(s/t)^2s^{-2} + C(C_1\vep)^2(s/t)^{-1}s^{-1}t^{-2} + C(C_1\vep)^2t^{-2}
\\
\leq& \underline{C(C_1\vep)^2(s/t)^2s^{-2}} 
\leq C(C_1\vep)^2t^{-3/2}(t-r)^{-1/2}.
\endaligned
\end{equation}
\begin{equation}\label{eq5-20-08-2020}
\aligned
|A_4^{\alpha}v\del_{\alpha}u|_{N-4}\leq& C|v|_{N-4}|\del u|_{N-6} + C|v|_{N-6}|\del u|_{N-4}
\\
\leq& C(C_1\vep)^2(s/t)^2s^{-2} + C(C_1\vep)^2(s/t)^3s^{-2} + C(C_1\vep)^2s^{-3}
\\
\leq& \underline{C(C_1\vep)^2\big((s/t)^2s^{-2} + s^{-3}\big)}
\leq  C(C_1\vep)^2t^{-3/2}(t-r)^{-1/2},
\endaligned
\end{equation}
where \eqref{eq1-31-07-2020} is applied on $|v|$.

The pure Klein-Gordon terms are bounded directly. We only write the following bound:
\begin{equation}\label{eq6-20-08-2020}
|\del v \del v|_{N-4}\leq C|v|_{N-3}^2\leq \underline{C(C_1\vep)^2(s/t)^2s^{-2}}
\leq C(C_1\vep)^2t^{-3/2}(t-r)^{-1/2}.
\end{equation}

Then we regard the right-hand-side of \eqref{eq12-30-07-2020}. For the convenience of expression, we denote by
$G_1 = F_1 - A_1^{\alpha\beta}\del_{\alpha}u\del_{\beta}u$ which are the terms in $F_1$ other than $A_1$. Then by \eqref{eq7-20-08-2020} - \eqref{eq6-20-08-2020},
\begin{equation}\label{eq0-08-09-2020}
|G_1|_{N-4}\leq C(C_1\vep)^2(s/t)^2s^{-2}.
\end{equation}
The last term is bounded as following:
$$
|uF_1|_{N-4} \leq |\phi F_1|_{N-4} + \frac{|A_1^{00}|}{2}|u|_{N-4}|uF_1|_{N-4}.
$$
Recalling \eqref{eq6-15-06-2020}, 
$$
|uF_1|_{N-4}\leq |\phi F_1|_{N-4} + CC_1\vep |uF_1|_{N-4}
$$
so when $\vep$ sufficiently small such that 
\begin{equation}\label{eq1-08-09-2020}
|A_1^{00}u|_{N-4}\leq 1,
\end{equation}
then
$$
\aligned
|uF_1|_{N-4}\leq& C|\phi F_1|_{N-4}
\\
\leq& C(C_1\vep)^2(s/t)^2s^{-2}
\sup_{\Hcal^{\text{ext}}_{[2,s_1]}}
\big\{(s/t)^{-1}|\phi|_{N-4}\big\}
\\
& + C(C_1\vep)^2(s/t)s^{-2}\sup_{\Hcal^{\text{int}}_{[2,s_1]}}\{(s/t)^{-1}|\phi|_{N-4}\} \mathbbm{1}_{\Hcal^{\text{int}}_{[2,s_1]}}.
\endaligned
$$
Apply \eqref{eq5-14-09-2020} on the last term and remark that in $\Hcal^\text{int}_{[2,s_1]}$, $ 0< \sqrt{19}/10 \leq  (s/t) < 1$, 
\begin{equation}\label{eq2-08-09-2020}
|uF_1|_{N-4}\leq C(C_1\vep)^2(s/t)^2s^{-2} + C(C_1\vep)^2(s/t)^2s^{-2}
\sup_{\Hcal^{\text{ext}}_{[2,s_1]}}
\big\{(s/t)^{-1}|\phi|_{N-4}\big\}.
\end{equation}

Now  we apply Proposition \ref{prop1-01-06-2020} on  
$$
\Box \del^IL^J \phi = -A_1^{00}\del^IL^J (u F_1) + \del^IL^J G_1,\quad |I|+|J|\leq N-4.
$$
We do the following decomposition. For $|I| + |J|\leq N-4$,
$$
\del^IL^J \phi = \phi^{IJ}_I + \phi^{IJ}_F
$$
with
\begin{equation}\label{eq6-30-07-2020}
\aligned
&\Box \phi^{IJ}_I = 0,
\\
&\phi^{IJ}(2,x) = \del^IL^J \phi(2,x),\quad \del_t u(2,x) = \del_t\del^IL^J\phi(2,x)
\endaligned
\end{equation}
and 
\begin{equation}\label{eq5-30-07-2020}
\Box \phi^{IJ}_F = \del^IL^J\big( -A_1^{00} uF_1 + G_1\big),\quad \phi^{IJ}_F(2,x) = \del_t \phi^{IJ}_F(2,x) = 0. 
\end{equation}
Then $\del^IL^J \phi = \phi^{IJ}_I + \phi^{IJ}_F$.
Remark that $\del^IL^J \phi(2,x), \del_t\del^IL^J\phi(2,x)$ are compactly supported in $\{|x|<1\}$. Furthermore, recall the relation of $\phi$ and $u$, there is a constant $K_N$ determined by $N$ and the system such that for $|I|+|J|\leq N-4$,
\begin{equation}\label{eq1-14-09-2020}
\aligned
\|\del^IL^J \phi(2,\cdot)\|_{L^{\infty}(\RR^2)} + \|\del_x\del^IL^J \phi(2,\cdot)\|_{L^\infty(\RR^2)} &+
\| \del_t\del^IL^J\phi(2,\cdot)\|_{L^{\infty}(\RR^2)}
\\
\leq& K_N \big(\|u_0\|_{H^{N+1}} + \|u_1\|_{H^N}\big)
\\
\leq& K_N\vep.
\endaligned
\end{equation}
Then by Lemma \ref{lem1-13-06-2020}
\begin{equation}
|\phi^{IJ}_I|_{N-4}\leq \bar{C}K_N\vep s^{-1}
\end{equation}
with $\bar{C}$ a universal constant. 

On the other hand, by Lemma \ref{lem3 Poisson} applied on \eqref{eq5-30-07-2020} with $\mu = \nu=1/2$ and $(t,x)\in \Hcal^{\text{ext}}_{[2,s_1]}$, recall \eqref{eq2-08-09-2020} and \eqref{eq0-08-09-2020}
\begin{equation}\label{eq9-30-07-2020}
|\phi^{IJ}_F|_{N-4}(t,x)\leq C(C_1\vep)^2(s/t)\sup_{\Hcal^{\text{ext}}_{[2,s_1]}}\big\{(s/t)^{-1}|\phi|_{N-4}\big\} + C(C_1\vep)^2(s/t).
\end{equation}
Apply \eqref{eq9-30-07-2020} at each point of $\Hcal^{\text{ext}}_{[2,s_1]}$, for $|I|+|J|\leq N-4$,
$$
\aligned
&\sup_{\Hcal^{\text{ext}}_{[2,s_1]}}\big\{(s/t)^{-1}|\phi^{IJ}_F|\big\}
\\
\leq &C(C_1\vep)^2(s/t)\sup_{\Hcal^{\text{ext}}_{[2,s_1]}\atop |I'|+|J'|\leq N-4}\big\{(s/t)^{-1}|\phi^{I'J'}_F|\big\} 
+ C(C_1\vep)^2(s/t)\sup_{\Hcal^{\text{ext}}_{[2,s_1]}\atop |I'|+|J'|\leq N-4}\{(s/t)^{-1}|\phi^{I'J'}_I|\} 
\\
&+ C(C_1\vep)^2(s/t).
\endaligned
$$
This leads to 
$$
\aligned
\sup_{\Hcal^{\text{ext}}_{[2,s_1]}\atop|I|+|J|
\leq N-4}\big\{(s/t)^{-1}|\phi^{IJ}_F|\big\}
\leq& C(C_1\vep)^2\sup_{\Hcal^{\text{ext}}_{[2,s_1]}\atop |I|+|J|\leq N-4}\big\{(s/t)^{-1}|\phi^{IJ}_F|\big\} 
\\
&+ C(C_1\vep)^2K_N\vep s^{-1} + C(C_1\vep)^2(s/t).
\endaligned
$$
Then when $\vep$ sufficiently small such that
\begin{equation}\label{eq3-08-09-2020}
C(C_1\vep)^2\leq \frac{1}{2}, \quad K_N\vep \leq 1,
\end{equation}
we conclude that
\begin{equation}
\sup_{\Hcal^{\text{ext}}_{[2,s_1]}\atop |I|+|J|\leq N-4}\big\{(s/t)^{-1}|\phi^{IJ}_F|\big\}\leq C(C_1\vep)^2.
\end{equation}

So we conclude:
\begin{equation}\label{eq2-20-08-2020}
|\phi|_{N-4}\leq \max_{|I|+|J|\leq N-4}\{|\phi^{IJ}_I| + |\phi^{IJ}_F|\} \leq \bar{C}K_N\vep s^{-1} + C(C_1\vep)^2(s/t), \quad \text{in}\quad \Hcal^{\text{ext}}_{[2,s_1]}.
\end{equation} 
Now remark the estimate 
$$
|\dels\dels \del^IL^J \phi|\leq Ct^{-2}|\phi|_{N-4}\leq C\bar{C}K_N\vep (s/t)^2s^{-3} + C(C_1\vep)^2(s/t)^3s^{-2}
$$
for $|I|+|J|\leq N-6$ with $C$ determined by $N$. Then \eqref{eq1-30-07-2020} is concluded in $\Hcal^{\text{ext}}_{[2,s_1]}$ and thus valid on $\Hcal_{[2,s_1]}$. 

Based on \eqref{eq2-20-08-2020} and \eqref{eq5-14-09-2020}, we also conclude the bound on $|u|_{N-4}$:
\begin{equation}\label{eq1-15-09-2020}
|u|_{N-4}\leq C_3\vep s^{-1} + C (C_1\vep)^2(s/t)
\end{equation}
with $C_3$ a constant determined by the system and $N$. Suppose that $C_1\geq C_3$, we obtain
\begin{equation}\label{eq2-14-09-2020}
|u|_{N-4}\leq CC_1\vep (s/t).
\end{equation}
 
%

\subsection{Proof of \eqref{eq2-30-07-2020}}
To do so we need to bound each term in right-hand-side of \eqref{eq12-30-07-2020}. Thanks to \eqref{eq1-31-07-2020}, the pure Klein-Gordon terms are bounded in a trivial manner. We only write the bound on $B_2$:
$$
|\del v \del v|_{N-6}\leq |v|_{N-5}^2\leq C(C_1\vep)^2(s/t)^4s^{-2}.
$$

The bound of mixed term $A_3$ is also based on \eqref{eq1-31-07-2020}. Recall \eqref{eq13-10-06-2020},
$$
\aligned
&|A_3^{\alpha\beta}\del_{\alpha}u\del_{\beta}v|_{N-6}
\\
\leq& C(s/t)^2|\del u|_{N-6}|\del v|_{N-6} + |\dels u|_{N-6}|\del v|_{N-6} +|\del u|_{N-6}|\dels v|_{N-6}
\\
\leq& C(C_1\vep)^2(s/t)^4s^{-2} + C(C_1\vep)^2(s/t)^3s^{-2} + C(C_1\vep)^2(s/t)^2s^{-3}
\\
\leq& C(C_1\vep)^2(s/t)^3s^{-2}.
\endaligned
$$

For the term $A_4$, by \eqref{eq1-31-07-2020} (the bound of order $N-6$)
$$
|A_4^{\alpha}v\del_{\alpha}u|_{N-6}\leq C|v|_{N-6}|\del u|_{N-6}\leq C(C_1\vep)^2(s/t)^4s^{-2} + C(C_1\vep)^3(s/t)s^{-3}.
$$

To bound $|uF_1|_{N-6}$, we need the following bound:
\begin{equation}\label{eq1-20-08-2020}
|\dels u|_{N-5}\leq CC_1\vep (s/t)^2s^{-1}.
\end{equation}
Compared with \eqref{eq5-06-05-2020}, the regularity is improved by two order. This is a direct result from \eqref{eq2-14-09-2020}. 

Equipped with \eqref{eq1-20-08-2020}, we make the following bound on $F_1$
\begin{equation}\label{eq3-20-08-2020}
|F_1|_{N-6}\leq C(C_1\vep)^2(s/t)^2s^{-2} + C(C_1\vep)s^{-3}.
\end{equation}
To prove this, recall the underlined bounds in \eqref{eq4-20-08-2020}, \eqref{eq5-20-08-2020} and \eqref{eq6-20-08-2020}. In fact we only need to improve the bound on $A_1$ for order $N-6$. Thanks to \eqref{eq1-20-08-2020},
$$
\aligned
|A_1^{\alpha\beta}\del_{\alpha}u\del_{\beta}u|_{N-6}
\leq& C(s/t)^2|\del u|_{N-6}^2 + C|\dels u|_{N-6}|\del u|_{N-6}
\\
\leq& C(C_1\vep)^2(s/t)^2s^{-2}.
\endaligned
$$
This concludes \eqref{eq3-20-08-2020}. Then thanks to \eqref{eq2-20-08-2020}
$$
|uF_1|_{N-6}\leq C|u|_{N-6}|F_1|_{N-6}\leq C(C_1\vep)^2(s/t)^3s^{-2} +  C(C_1\vep)^2(s/t)s^{-3}.
$$

Recall the definition of  $S^w[\del^IL^J\phi]$,  we conclude by \eqref{eq2-30-07-2020}.

\subsection{Conclusion of this section}
We recall Proposition \ref{prpo2 wave-sharp} and apply it on \eqref{eq1-21-08-2020}. Thanks to  \eqref{eq1-30-07-2020} and \eqref{eq2-30-07-2020}, for a point $(\bar{t},\bar{x})\in \Hcal_{[2,s]}$. 
$$
\aligned
|\bar{s}\del_t\del^IL^J\phi (\bar{t},\bar{x})|
\leq& 2\sqrt{2}\|\del_t \del^IL^J \phi\|_{L^{\infty}(\Hcal_2)} 
\\
&+ \sqrt{2}\Big|\int_{2}^{\bar{t}} W_{\bar{t},\bar{x}}[\del^IL^J \phi](t)e^{-\int_t^{\bar{t}}P_{\bar{t},\bar{x}}(\eta)d\eta}dt\Big|.
\endaligned
$$
Remark that $P(t,r)\geq \frac{1}{4}(s/t)^2t^{-1}$, then
$$
\aligned
0\leq& \int_{2}^{\bar{t}} W_{t,x}[\del^IL^J \phi](t)e^{-\int_t^{\bar{t}} P_{\bar{t},\bar{x}}(\eta)d\eta}dt
\\
\leq&  4\big(C_2\vep + C(C_1\vep)^2\big)\int_{2}^{\bar{t}} \frac{1}{4}(s/t)^2t^{-1}\Big|_{\gamma(t;\bar{t},\bar{x})} 
e^{-\frac{1}{4}\int_{t}^{\bar{t}} (s/t)^2t^{-1}\big|_{\gamma(\eta,\bar{t},\bar{x})} d\eta} d\tau
\\
&+ C(C_1\vep)^2\int_2^{\bar{t}}(s/t)s^{-2}\bigg|_{\gamma(t;\bar{t},\bar{x})}dt
\\
\leq& C(C_1\vep)^2 + 4C_2\vep + C(C_1\vep)^2\int_2^{\bar{t}}t^{-3/2}dt
\\
\leq& C(C_1\vep)^2 + 4C_2\vep.
\endaligned
$$
Thus
$$
|\bar{s}\del_t\del^IL^J \phi(t,x)|\leq C(C_1\vep)^2 + 4C_2\vep
$$
where we have applied that $|\del^IL^J (u^2)|\leq C(C_0\vep)^2\leq C(C_1\vep)^2$ on $\Hcal_2$. Then we obtain:
\begin{equation}\label{eq3-31-07-2020}
|s\del_t\del^IL^J u|\leq 4C_2\vep + C(C_1\vep)^2, \quad |I|+|J|\leq N-6.
\end{equation}

\section{Conclusion of the bootstrap argument}
\label{sec-conclusion}
We consider \eqref{eq9-15-06-2020}, \eqref{eq11-15-06-2020}, \eqref{eq8-29-07-2020} together with \eqref{eq3-31-07-2020}. We take $C_1$ sufficiently large with 
$$
C_1\geq C_3 + 8C_2 + 2C_0
$$ 
where $C_3$ is introduced in \eqref{eq1-15-09-2020}. Then taking
\begin{equation}\label{eq-last}
0\leq \vep\leq \frac{C_1 - 2C_0 - 8C_2}{2CC_1^2},
\end{equation}
where $C$ is a constant determined by $N$ and the system, such that \eqref{eq1-08-09-2020} and \eqref{eq3-08-09-2020} hold, 
the improved bounds \eqref{eq1'-26-04-2020}, \eqref{eq2'-26-04-2020} and \eqref{eq3'-02-06-2020} are guaranteed.

\section{A glance at Klein-Gordon-Zakharov model system} 
\label{sec-Zakharov}
We recall the Klein-Gordon-Zakharov model system \eqref{eq1-Zakharov}. 
In this section we will explain more on its Hessian structure. One may compare this structure with \eqref{eq-main}. To make it more clear, we consider the following general system
\begin{equation}\label{eq7-15-08-2020}
\aligned
&\Box u = v^2,
\\
&\Box v + v = vP^{\alpha\beta}\del_{\alpha}\del_{\beta}u
\endaligned
\end{equation}
with $P^{\alpha\beta}$ constants. Suppose that the initial data is imposed on $\Hcal_2$, compactly supported and sufficiently regular. We make the following bootstrap assumption on a time interval $[2,s_1]$:
\begin{equation}\label{eq8-15-08-2020}
\Ecal_0^N(s,\del_{\alpha} u)^{1/2} + \Ecal_0^N(s,u)^{1/2} + \Ecal_{0,c}^N(s,v)^{1/2}\leq C_1\vep s^{\delta},
\end{equation}
\begin{equation}\label{eq9-15-08-2020}
\Ecal_{0,c}^{N-1}(s,v)^{1/2}\leq C_1\vep.
\end{equation}
Here remark that we need to bound the energy on $\del u$, i.e., in our framework $u$ enjoys one more order of regularity than $v$. That is also why we demand one more regularity on the initial data of $u$ than that of $v$ in Theorem \ref{main2-thm}.

By Klainerman-Sobolev type inequality, 
\begin{equation}\label{eq10-15-08-2020}
s|\del u|_{N-2} + t|\dels u|_{N-2} \leq CC_1\vep s^{\delta},
\end{equation}
\begin{equation}\label{eq11-15-08-2020}
s|\del v|_{N-3} + t|\dels v|_{N-3} + t|v|_{N-3}\leq CC_1\vep.
\end{equation}
With \eqref{eq11-15-08-2020} and the following relation:
$$
\Box \del_{\alpha}u = 2v\del_{\alpha}v,
$$
one can easily establish the following improved energy bound on wave component:
\begin{equation}\label{eq18-15-08-2020}
\Ecal_0^N(s,\del_{\alpha} u)^{1/2} + \Ecal_0^N(s,u)^{1/2}\leq C_0\vep + C(C_1\vep)^2s^{\delta}
\end{equation}
where $C_0\vep$ measures the initial energies. 

The bounds on Klein-Gordon component depend on the Hessian structure $vP^{\alpha\beta}\del_{\alpha}\del_{\beta}u$.  By Proposition \ref{prop1-14-08-2020}, one can establish the following bounds on Hessian form:
\begin{equation}\label{eq12-15-08-2020}
\aligned
\|(s/t)^2s|\del\del u|_{N-1}\|_{L^2(\Hcal_s)}\leq& C\|(s/t)|\del u|_N\|_{L^2(\Hcal_s)} + C\|s|v^2|_{N-1}\|_{L^2(\Hcal_s)}
\\
\leq& CC_1\vep s^{\delta},
\endaligned
\end{equation}
\begin{equation}\label{eq13-15-08-2020}
(s/t)^2|\del\del u|_{N-3}\leq Ct^{-1}|\del u|_{N-2} + |v^2|_{N-3}\leq CC_1\vep (s/t)s^{-2+\delta}.
\end{equation}
Here we remark that the Hessian form enjoys better principle decay ($-2+\delta$ order) than the gradient ($-1+\delta$ order).
We also need to recover some conical decay on $v$. These are done by applying Proposition \ref{prop1-fast-kg}. We first remark that
$$
|v|_{N-4}\leq CC_1\vep (s/t)^2|\del v|_{N-3} + C|v\del\del u|_{N-4}\leq CC_1\vep (s/t)^2s^{-1} + CC_1\vep |v|_{N-4}.
$$
Taking $\vep$ sufficiently small such that $1-CC_1\vep \geq 1/2$, one obtains
\begin{equation}\label{eq16-15-08-2020}
|v|_{N-4}\leq CC_1\vep(s/t)^2s^{-1}.
\end{equation}
On the other hand, we remark that
$$
\aligned
&\|(s/t)^{-1}|v|_{N-1}\|_{L^2(\Hcal_s)}
\\
\leq& C\|(s/t)|\del v|_N\|_{L^2(\Hcal_s)} + C\|(s/t)^{-1}|v\del \del u|_{N-1}\|_{L^2(\Hcal_s)}
\\
\leq& CC_1\vep s^{\delta}
+ C\|(s/t)^{-1}|v|_{N-4}|\del\del u|_{N-1}\|_{L^2(\Hcal_s)} 
\\
&+ C\|(s/t)^{-1}|\del\del u|_{N-3}|v|_{N-1}\|_{L^2(\Hcal_s)}
\\
\leq& CC_1\vep s^{\delta} + CC_1\vep s^{-1}\|(s/t)|\del\del u|_{N-1}\|_{L^2(\Hcal_s)} + CC_1\vep\|(s/t)^{-1}|v|_{N-1}\|_{L^2(\Hcal_s)}
\endaligned
$$
where for the last inequality \eqref{eq16-15-08-2020} is applied. Now taking $\vep$ sufficiently small, the following bound is established:
\begin{equation}\label{eq17-15-08-2020}
\|(s/t)^{-1}|v|_{N-1}\|_{L^2(\Hcal_s)}\leq CC_1\vep s^{\delta}.
\end{equation}

Taking \eqref{eq12-15-08-2020}, \eqref{eq13-15-08-2020} together with \eqref{eq16-15-08-2020} and \eqref{eq17-15-08-2020}, the $L^2$ norm of source terms of Klein-Gordon equation is bounded as following:
$$
\aligned
&\||v\del\del u|_{N-1}\|_{L^2(\Hcal_s)}
\\
\leq& C\||v|_{N-4}|\del\del u|_{N-1}\|_{L^2(\Hcal_s)} + C\||v|_{N-1}|\del\del u|_{N-3}\|_{L^2(\Hcal_S)}
\\
\leq& CC_1\vep s^{-1}\|(s/t)^2|\del\del u|_{N-1}\|_{L^2(\Hcal_s)} + CC_1\vep s^{-2+\delta}\|(s/t)^{-1}|v|_{N-1}\|_{L^2(\Hcal_s)}
\\
\leq& C(C_1\vep)s^{-2+2\delta}.
\endaligned
$$
This is integrable with respect to $s$ (and remark that there is a margin). So we obtain, thanks to energy estimate,
\begin{equation}\label{eq19-15-08-2020}
\Ecal_{0,c}^{N-1}(s,v)^{1/2}\leq C_0\vep + C(C_1\vep)^2.
\end{equation}

The improvement on the leading order energy of $v$ is the most difficult. In \cite{Dong-2020-2} a type weighted energy estimate (called the ghost weight) on wave component is developed to overpass this. In our context due to Proposition \ref{prpo2 wave-sharp}, we can establish the following bound:
\begin{equation}\label{eq14-15-08-2020}
|\del_{\alpha}\del_{\beta}u|_{N-4}\leq CC_1\vep s^{-1}.
\end{equation}
Once this is done, we can establish the following bound on source terms:
$$
\aligned
\||v\del_{\alpha}\del_{\beta}u|_N\|_{L^2(\Hcal_s)}\leq& CC_1\vep\|s^{-1}|v|_N\|_{L^2(\Hcal_s)} 
+ CC_1\vep \|t^{-1}|\del_{\alpha}\del_{\beta} u|_N\|_{L^2(\Hcal_s)}
\\
\leq& C(C_1\vep)^2s^{\delta} + CC_1\vep s^{-1}\|(s/t)|\del(\del_{\alpha} u)|_N\|_{L^2(\Hcal_s)}
\\
\leq& C(C_1\vep)^2s^\delta.
\endaligned
$$
Then by energy estimate, we obtain
\begin{equation}\label{eq20-15-08-2020}
\Ecal_{0,c}^N(s,v)^{1/2}\leq C_0\vep  + C(C_1\vep)^2s^{\delta}.
\end{equation}

In order to establish \eqref{eq14-15-08-2020}, one only need to apply Proposition \ref{prpo2 wave-sharp}. Remark that for $|I|+|J|\leq N-4$,
$$
\aligned
|\Delta^w[\del^IL^J \del_{\alpha}u]|& \leq Cs|\dels\dels \del_{\alpha} u|_{N-4}
\leq Cst^{-2}|\del u|_{N-2}
\leq C(s/t)t^{-1}s^{-1+\delta}
\\
&\leq CC_1\vep t^{-3/2+\delta/2},
\endaligned
$$
$$
|S^w[\del^IL^J\del_{\alpha}u]|_{N-4}\leq Cs|v\del_{\alpha}v|_{N-4}\leq CC_1\vep (s/t)^2t^{-1}\leq CC_1\vep \frac{t-r}{t}t^{-1}
$$
where \eqref{eq4 notation}  and \eqref{eq16-15-08-2020} are applied on $|\dels\dels \del u|_{N-4}$ and $|v|_{N-4}$ respectively. Remark that, $S^w$ can be bounded by $P(t,r)$. One may compare this with \eqref{eq1-30-07-2020}. Then apply Proposition \ref{prpo2 wave-sharp} on 
$$
\Box \del^IL^J \del_{\alpha}u = 2\del^IL^J (v\del_{\alpha} v)
$$
and substitute these bounds in to the right-hand-side of \eqref{eq1-29-05-2020}, \eqref{eq14-15-08-2020} is established.

Now in \eqref{eq18-15-08-2020}, \eqref{eq19-15-08-2020} and \eqref{eq20-15-08-2020} we take $\vep\leq \frac{C_1-2C_0}{2CC_1}$, the bootstrap assumptions are improved and this concludes the bootstrap argument.

\appendix

\section{Proof of Lemma \ref{prop1-30-07-2020}}
This is by the following observation. In $\RR^{2+1}$, let
$$
\xi = (1,1,0)^{T},\quad,\bar{\xi} = (1,-1,0)^{T},\quad \eta = (1,0,1)^{T}, \quad \bar{\eta} = (1,0,-1)^T. 
$$
Then $\xi,\eta$ are null vectors. Then 
$$
A(\xi,\xi) = A(\bar{\xi},\bar{\xi}) = 0.
$$
This leads to
$$
A^{00} + A^{11} + 2A^{01} = A^{00} + A^{11} - 2A^{01}\Rightarrow A^{01} = A^{10} = 0.
$$
In the same manner with $\eta$ and $\bar{\eta}$ we observe that $A^{20} = A^{02} = 0$. On the other hand,
$A(\xi,\xi) = 0$ leads to
$$
A^{00} + A^{11} = 0\Rightarrow A^{11} = -A^{00}.
$$
In the same manner, $A^{22} = -A^{00}$. 

Finally, in order to fix $A^{12}$, we consider the null vector $\theta = (\sqrt{2},1,1)$. Remark that $A^{10} = A^{20} = 0$
$$
0 = A(\theta,\theta) = 2A^{00} + A^{11} + A^{22} + 2A^{12}\Rightarrow A^{12} = 0.
$$ 
So we conclude that
$$
A^{\alpha\beta} = A^{00}m^{\alpha\beta}. 
$$

\section{Basic notation and calculus within hyperboloidal foliation}\label{sec-appendix-C}
\subsection{Sketch on the proof of \eqref{eq11-10-06-2020}}
In this subsection all constants are determined by the order of the derivatives except otherwise specified.

The first bound is based on the following decomposition of commutator in $\Kcal$:
\begin{equation}\label{eq14-10-06-2020}
[L^J,\del^I] 
= \sum_{|I'| = |I|\atop |J'|<|J|}\Gamma^{JI}_{I'J'}\del^{I'}L^{J'}
\end{equation}
with $\Gamma^{IJ}_{I'J'}$ constants. This can be easily proved by induction on $|I|$ and $|J|$. 

Then let $Z^K$ be of type $(p-k,k,0)$, i.e., it contains at most $(p-k)$ partial derivatives and $k$ boosts. Then it can be written as following:
$$
Z^K = \del^{I_1}L^{J_1}\cdots \del^{I_r}L^{J_r}
$$
where $I_1$ and $L_r$ may be empty indices. By commuting $\del^{I_k}$ with $L^{J_{k-1}}$, we arrive at the following decomposition:
\begin{equation}\label{eq15-10-06-2020}
Z^K = \sum_{|I| \leq p-k\atop  |J|\leq k}\Gamma^K_{IJ}\del^IL^J 
\end{equation}
with $\Gamma^K_{IJ}$ constants.

Finally, let us consider $Z^K\del_{\alpha}u$. By the above decomposition, it is a finite linear combination of the terms $\del^IL^J\del_{\alpha}u$ with $|I|\leq p-k, |J|\leq k$. Then by \eqref{eq14-10-06-2020} with $I=(\alpha)$ we commute $L^J$ and $\alpha$, then the first bound of \eqref{eq11-10-06-2020} is established.

For the second bound, we need to recall the following property on the function $(s/t) = \frac{\sqrt{t^2-r^2}}{t}$. We can prove (see for example in \cite{LM1}) that 
\begin{equation}\label{eq16-10-06-2020}
\big|\del^IL^J(s/t)\big|\leq \left\{
\aligned
&C(s/t), \quad &&|I|=0,
\\
&Cs^{-1}\leq C(s/t),\quad &&|I|>0.
\endaligned
\right.
\end{equation}
This is also proved by induction on $|I|$ and $|J|$. Then apply \eqref{eq15-10-06-2020} on $Z^K\big((s/t)\del_{\alpha} u\big)$
$$
Z^K((s/t)\del_{\alpha}u) = \sum_{K_1+K_2=K}Z^{K_1}(s/t)Z^{K_2}u.
$$
The first factor is bounded by $(s/t)$. The second factor is bounded by the first bound of \eqref{eq11-10-06-2020}. Then we conclude by the second bound of \eqref{eq11-10-06-2020}.
\subsection {Sketch on the proof of \eqref{eq3 notation}, \eqref{eq4 notation} and \eqref{eq1 lem5 notation}}
Remark that by \eqref{eq15-10-06-2020}, we only need to bound $\del^IL^J\delu_a u$. Then remark that
$$
\del^IL^J \delu_a u = \del^IL^J(t^{-1}L_au) = \sum_{I_1+I_2=I\atop J_1+J_2=J}\del^{I_1}L^{J_1}(t^{-1})\del^{I_2}L^{J_2}L_au.
$$
The first factor is bounded by $Ct^{-1}$ in $\Kcal$ (can be observed by homogeneity or induction). For the second factor, when $|I_2| = 0$, remark that 
$$
t^{-1}L^{J_2}L_au = t^{-1}L_{a'}L^{J'}u = \delu_{a'}L^{J'}u,\quad |J|\geq |J_2| = |J_2'| \geq 0.
$$
When $|I_2|\geq 1$,  we write $L^{J_2}L_a = L^{J_2'} = L_bL^{J_2''}, |J_2''| = |J_2|\geq 0$. Then 
$$
\del^{I_2}L^{J_2}L_au = \del^{I_2}L_bL^{J_2''} = L_b\del^{I_2}L^{J_2''}u + [\del^{I_2},L_b]L^{J_2''}u
$$
Then
$$
|\del^{I_1}L^{J_1}(t^{-1})\del^{I_2}L^{J_2}L_au|\leq C| t^{-1}L_b\del^{I_2}L^{J_2''}u| + | [\del^{I_2},L_b]L^{J_2''}u|.
$$
Then by the relation $t^{-1}L_b = \delu_b$ and \eqref{eq14-10-06-2020} applied on $[\del^{I_2},L_b]$, \eqref{eq3 notation} is established.

\eqref{eq4 notation} is by applying twice the above argument. 

For \eqref{eq1 lem5 notation}, we only need to remark that
$$
\del_{\alpha}\delu_au = t^{-1}\del_{\alpha}L_au - t^{-1}\delta_{\alpha}^0\delu_au,
$$
then by homogeneity, we conclude the desired result.

\subsection{Sketch on \eqref{eq13-10-06-2020}}
\label{subsec-null}
We need to recall the following result on null form. Let $A^{\alpha\beta}$ be a constant coefficient quadratic form satisfying the null condition. Then
\begin{equation}\label{eq1-11-06-2020}
|\Au^{00}|_{p,k}\leq C(s/t)^2.
\end{equation}
This is proved in \cite{LM1}.
In fact one only need to check $\del^IL^J \Au^{00}$. Remark that 
$$
\Au^{00} = A^{\alpha\beta}\Psiu_{\alpha}^0\Psiu_{\beta}^0
$$
and
$$
\big(\Psiu_{\alpha}^0\big)_{\alpha=0,1,2} = (1,-x^1/t,-x^2/t)^T = \big(1-(r/t), 0,0\big)^T + \eta
$$
where $\eta = \big((r/t), -x^1/t,-x^2/t\big)^T$ is null. Denote by $\xi = (1,-x^1/t,-x^2/t)^T$, then
$$
\aligned
\Au^{00} =& A(\xi+\eta,\xi+\eta) = A(\xi,\xi) + A(\eta,\xi) + A(\xi,\eta) + A(\eta,\eta) 
\\
=& A(\xi,\xi) + A(\eta,\xi) + A(\xi,\eta) 
\endaligned
$$
where for the last equality the null condition of $A$ is applied on the last term. Then we check the rest terms. Remark that
$$
A(\xi,\xi) = \Big(\frac{t-r}{t}\Big)^2A^{00},\quad A(\xi,\eta) = \frac{t-r}{t}\Big(A^{00} - \frac{x^a}{t}A^{a0}\Big).
$$
The factor $\frac{t-r}{t}$ supplies the conical decay $(s/t)^2 = \frac{(t+r)(t-r)}{t^2}$. This shows the bound \eqref{eq1-11-06-2020} at zero order

For  higher order, remark that this bound is non-trivial only when $r\geq t/2$ and one can check this directly by induction on $|I|,|J|$. 

Then remark that
$$
A^{\alpha\beta}\del_{\alpha}u\del_{\beta}u = \Au^{00}\del_tu\del_tu + 2\Au^{a0}\delu_au\del_tu + \Au^{ab}\delu_au\delu_b u.
$$
Then differentiate the above identity with respect to $\del^IL^J$, and remark that $|\Au^{\alpha\beta}|_{p,k}\leq C|A|$ because these components are homogeneous functions. Then apply \eqref{eq1-11-06-2020}, the desired bound is established. 

\bibliographystyle{elsarticle-num}
\bibliography{WKGm-bibtex}





%
%
%
\end{document}